\documentclass[a4paper,twoside,english,BCOR7.5mm]{article}

\usepackage[T1]{fontenc}
\usepackage[utf8]{inputenc}
\usepackage{color}
\definecolor{note_fontcolor}{rgb}{0, 0, 1}
\usepackage{babel}
\usepackage{verbatim}
\usepackage{amsmath}
\usepackage{amsthm}
\usepackage{amssymb}
\usepackage{esint}
\usepackage[all]{xy}
\usepackage[unicode=true,
 bookmarks=true,bookmarksnumbered=true,bookmarksopen=true,bookmarksopenlevel=2,
 breaklinks=false,pdfborder={0 0 1},backref=false,colorlinks=true]
 {hyperref}

\makeatletter

\pdfpageheight\paperheight
\pdfpagewidth\paperwidth

\newenvironment{lyxgreyedout}
  {\textcolor{note_fontcolor}\bgroup\ignorespaces}
  {\ignorespacesafterend\egroup}

\theoremstyle{plain}
\newtheorem{thm}{\protect\theoremname}
  \theoremstyle{definition}
  \newtheorem{problem}[thm]{\protect\problemname}
  \theoremstyle{plain}
  \newtheorem{cor}[thm]{\protect\corollaryname}
  \theoremstyle{definition}
  \newtheorem{defn}[thm]{\protect\definitionname}
  \theoremstyle{remark}
  \newtheorem{rem}[thm]{\protect\remarkname}
  \theoremstyle{definition}
  \newtheorem{example}[thm]{\protect\examplename}
  \theoremstyle{plain}
  \newtheorem{lem}[thm]{\protect\lemmaname}
  \theoremstyle{remark}
  \newtheorem{claim}[thm]{\protect\claimname}

\usepackage[all]{xy}

\newcommand{\xyR}[1]{
  \xydef@\xymatrixrowsep@{#1}}
\newcommand{\xyC}[1]{
  \xydef@\xymatrixcolsep@{#1}}

\newdir{|>}{!/4.5pt/@{|}*:(1,-.2)@^{>}*:(1,+.2)@_{>}}

\let\myTOC\tableofcontents
\renewcommand\tableofcontents{%
  \pdfbookmark[1]{\contentsname}{}
  \myTOC }

\def\LyX{\texorpdfstring{%
  L\kern-.1667em\lower.25em\hbox{Y}\kern-.125emX\@}
  {LyX}}

\usepackage{ifpdf}
\ifpdf

\IfFileExists{lmodern.sty}
 {\usepackage{lmodern}}{}

\fi 

\makeatother

  \providecommand{\claimname}{Claim}
  \providecommand{\corollaryname}{Corollary}
  \providecommand{\definitionname}{Definition}
  \providecommand{\examplename}{Example}
  \providecommand{\lemmaname}{Lemma}
  \providecommand{\problemname}{Problem}
  \providecommand{\remarkname}{Remark}
\providecommand{\theoremname}{Theorem}

\begin{document}

\title{On Gaussian multiplicative chaos}

\author{Alexander Shamov\thanks{E-mail address: \protect\href{mailto:trefoils@gmail.com}{trefoils@gmail.com}}}

\date{July 16, 2014\\
revised Mar 08, 2016}

\maketitle
\noindent %
\begin{lyxgreyedout}
\global\long\def\op#1{\operatorname{#1}}
 \global\long\def\midmid{\mathrel{}\middle|\mathrel{}}
 \global\long\def\NN{\mathbb{N}}
 \global\long\def\ZZ{\mathbb{Z}}
 \global\long\def\QQ{\mathbb{Q}}
 \global\long\def\RR{\mathbb{R}}
 \global\long\def\CC{\mathbb{C}}
 \global\long\def\HH{\mathbb{H}}
 \global\long\def\KK{\mathbb{K}}
 \global\long\def\FF{\mathbb{F}}
 \global\long\def\P{\op{\mathsf{P}}}
 \global\long\def\Q{\op{\mathsf{Q}}}
 \global\long\def\E{\op{\mathsf{E}}}
 \global\long\def\OMEGA{\op{\mathsf{\Omega}}}
 \global\long\def\I{\op{\mathsf{1}}}
 \global\long\def\Var{\op{\mathsf{Var}}}
 \global\long\def\H{\op{\mathsf{H}}}
 \global\long\def\Law{\op{\mathsf{Law}}}
 \global\long\def\Cov{\op{\mathsf{Cov}}}
 \global\long\def\Pc#1#2{\P\left\{  #1\midmid#2\right\}  }
 \global\long\def\Pcsq#1#2{\P\left[#1\midmid#2\right]}
 \global\long\def\Ec#1#2{\E\left[#1\midmid#2\right]}
 \global\long\def\Ecm#1#2#3{\E_{#1}\left[#2\midmid#3\right]}
 \global\long\def\Varc#1#2{\Var\left[#1\midmid#2\right]}
 \global\long\def\Hc#1#2{\H\left[#1\midmid#2\right]}
 \global\long\def\Lawc#1#2{\Law\left[#1\midmid#2\right]}
 \global\long\def\Lawcm#1#2#3{\Law_{#1}\left[#2\midmid#3\right]}
 \global\long\def\Covc#1#2{\Cov\left[#1\midmid#2\right]}
 \global\long\def\Wick#1{{:\!#1\!:}}
 \global\long\def\norm#1{\left\Vert #1\right\Vert }
 \global\long\def\const{\op{const}}
 \global\long\def\diam{\op{diam}}
 \global\long\def\diag{\op{diag}}
 \global\long\def\det{\op{det}}
 \global\long\def\per{\op{per}}
 \global\long\def\pf{\op{pf}}
 \global\long\def\sgn{\op{sgn}}
 \global\long\def\Leb{\op{Leb}}
 \global\long\def\tr{\op{tr}}
 \global\long\def\span{\op{span}}
 \global\long\def\supp{\op{supp}}
 \global\long\def\Spec{\op{Spec}}
 \global\long\def\rank{\op{rank}}
 \global\long\def\im{\op{im}}
 \global\long\def\coker{\op{coker}}
 \global\long\def\coim{\op{coim}}
 \global\long\def\colim{\op{colim}}
 \global\long\def\cone{\op{cone}}
 \global\long\def\cyl{\op{cyl}}
 \global\long\def\Hom{\op{Hom}}
 \global\long\def\Ext{\op{Ext}}
 \global\long\def\Tor{\op{Tor}}
 \global\long\def\RKHS{\op{RKHS}}
 \global\long\def\pr{\op{pr}}
 \global\long\def\id{\op{id}}
 \global\long\def\Sets{\op{\mathbf{Sets}}}
 \global\long\def\FinSets{\op{\mathbf{FinSets}}}
 \global\long\def\Groups{\op{\mathbf{Groups}}}
 \global\long\def\AbGroups{\op{\mathbf{AbGroups}}}
 \global\long\def\Mod{\op{\mathbf{Mod}}}
 \global\long\def\Modl#1{_{#1}\mathbf{Mod}}
 \global\long\def\Modr#1{\Mod_{#1}}
 \global\long\def\Modb#1#2{_{#1}\mathbf{Mod}_{#2}}
 \global\long\def\Top{\op{\mathbf{Top}}}
\end{lyxgreyedout}

\begin{abstract}
We propose a new definition of the Gaussian multiplicative chaos and
an approach based on the relation of subcritical Gaussian multiplicative
chaos to randomized shifts of a Gaussian measure. Using this relation
we prove general results on uniqueness and convergence for subcritical
Gaussian multiplicative chaos that hold for Gaussian fields with arbitrary
covariance kernels.


\end{abstract}
\tableofcontents{}

\section{Introduction}

\subsection{The object of interest}

Let $\left(\mathcal{T},\mu\right)$ be a finite measure space, and
let $X=\left(X\left(\omega,t\right)\right)_{\omega\in\OMEGA,t\in\mathcal{T}}$
be a Gaussian field parametrized by $t\in\mathcal{T}$ and defined
on a probability space $\left(\OMEGA,\P\right)$. With this data one
can associate the following random measure:
\begin{equation}
M\left(dt\right):=\exp\left[X\left(t\right)-\frac{1}{2}\E\left|X\left(t\right)\right|^{2}\right]\mu\left(dt\right).\label{eq:FormalGMC}
\end{equation}
The Gaussian multiplicative chaos (GMC) is the natural generalization
of such a random measure to the setting when the field $\left(X\left(t\right)\right)$
is defined in a distributional sense rather than pointwise, i.e. via
a family of formal ``integrals'' against test functions from an
appropriate class. Obviously, for such generalized Gaussian fields
\eqref{eq:FormalGMC} does not make sense literally, since $X\left(t\right)$
need not be well-defined as a random variable for any particular $t$.
Accordingly, in nontrivial cases $M$ is almost surely $\mu$-singular,
so the density $M\left(dt\right)/\mu\left(dt\right)$ is not well-defined
either.

The commonly used ways of interpreting \eqref{eq:FormalGMC} rigorously
and constructing such random measures proceed by approximating the
field $X$ by Gaussian fields $X_{n}$ that are, unlike $X$, defined
pointwise. One defines a GMC $M$ as a limit, in an appropriate topology,
of random measures
\begin{equation}
M_{n}\left(dt\right):=\exp\left[X_{n}\left(t\right)-\frac{1}{2}\E\left|X_{n}\left(t\right)\right|^{2}\right]\mu\left(dt\right).\label{eq:ApproxGMC}
\end{equation}
This approach leads naturally to the following problems, both of which
will be addressed in this paper.
\begin{problem}
Find conditions on the approximation $X_{n}\to X$ that are sufficient
for convergence of $M_{n}$. \label{prob:Convergence}
\end{problem}

\begin{problem}
Prove that the limit is independent of the approximation procedure.
\label{prob:Uniqueness}
\end{problem}
As far as we know, in previous works these problems have only been
partially solved under unnecessarily restrictive assumptions. Below
we provide an overview of the commonly used approximation procedures.

\subsubsection{Martingale approximation \cite{KahaneSur}}

The martingale approximation is employed in Kahane's original work
on GMC in \cite{KahaneSur}. In his construction the increments $X_{n}-X_{n-1}$
are independent (and $X_{0}:=0$), which implies that $\left(M_{n}\right)$
is a positive measure-valued martingale. The martingale property guarantees
that $M_{n}$ converges to a random measure $M$ in the sense that
$M_{n}\left[A\right]\to M\left[A\right]$ almost surely for any fixed
measurable set $A$. Moreover, $\E M=\mu$ iff the martingale $\left(M_{n}\left[\mathcal{T}\right]\right)$
is uniformly integrable, in which case the limit $M$ is taken as
the interpretation of \eqref{eq:FormalGMC}.

One intuitively expects the martingale approximation to yield the
``right'' and completely general notion of subcritical GMC. However,
Kahane's work falls short of establishing the basic setup in sufficient
generality.

The construction in \cite{KahaneSur} takes as its input a function
$K:\mathcal{T}\times\mathcal{T}\to\RR_{+}\cup\left\{ \infty\right\} $
--- thought of as the covariance kernel of the Gaussian field $X$.
$K$ is assumed to be decomposable into a sum
\begin{equation}
K\left(t,s\right):=\sum_{n}p_{n}\left(t,s\right)\label{eq:SigmaPositive}
\end{equation}
of kernels $p_{n}:\mathcal{T}\times\mathcal{T}\to\RR_{+}$ that are
both positive definite and \emph{positive in the pointwise sense}
(i.e. $p_{n}\left(t,s\right)\ge0,\forall t,s$). We may assume that
$K$ is finite $\mu\otimes\mu$-almost everywhere, but it may explode
on the diagonal. Unlike the sum $K$, the kernels $p_{n}$ are only
allowed to take finite values, so that indeed there are independent
Gaussian fields $X_{n}-X_{n-1}$ with covariance 
\[
\E\left(X_{n}\left(t\right)-X_{n-1}\left(t\right)\right)\left(X_{n}\left(s\right)-X_{n-1}\left(s\right)\right):=p_{n}\left(t,s\right).
\]
The kernels $K$ that admit a representation \eqref{eq:SigmaPositive}
(with $p_{n}$ continuous with respect to a given compact metrizable
topology on $\mathcal{T}$) are said to be of $\sigma$-positive type.
Under this $\sigma$-positivity assumption it is proved in \cite{KahaneSur}
that the \emph{law} of the limiting random measure $M$ is independent
of the decomposition \eqref{eq:SigmaPositive}. 

One problem with this approach is that due to the pointwise positivity
assumption $p_{n}\left(t,s\right)\ge0$, $\sigma$-positivity is both
unnecessarily restrictive and hard to check in practice. Another problem
is that while $M$ is naturally defined on the same probability space
with the underlying Gaussian randomness, Kahane only proves uniqueness
in law rather than uniqueness of $M$ as a function of the Gaussian
field. That is, his result does not rule out the possibility that
different decompositions of $X$ yield different random measures with
the same law.

\subsubsection{Mollifying operators \cite{RVRevisited,DuShKPZ,RVReview}}

This approximation technique restricts the generality to an important
special case where $\mathcal{T}$ is a domain in $\RR^{d}$, $\mu$
is the Lebesgue measure, and the covariance kernel $K$ of the field
$X$ has the special form
\begin{equation}
K\left(t,s\right):=\gamma^{2}\log^{+}\norm{t-s}^{-1}+g\left(t,s\right),\label{eq:LogKernel}
\end{equation}
where $\log^{+}:=\max\left(\log,0\right)$, the function $g:\mathcal{T}\times\mathcal{T}\to\RR$
is bounded and continuous, and $\gamma^{2}<2d$. In this case $X$
is obviously well-defined as a random distribution, i.e. it can be
integrated against smooth test functions. The fields $X_{n}$ are
constructed by convolution:
\begin{equation}
X_{n}\left(t\right):=\intop_{\mathcal{T}}X\left(t^{\prime}\right)\psi_{1/n}\left(t-t^{\prime}\right)dt^{\prime},\label{eq:Convolution}
\end{equation}
\[
\psi_{1/n}\left(x\right):=n^{d}\psi\left(nx\right),
\]
where $\psi:\RR^{d}\to\RR_{+}$, $\intop\psi\left(t\right)dt=1$,
subject to appropriate smoothness conditions. This approximation method
was used in \cite{RVRevisited} for stationary fields on $\mathcal{T}=\RR^{d}$,
and according to \cite{RVReview}, the same techniques apply to the
non-stationary setting. Unlike in Kahane's approach, for convolution
approximations the convergence of $M_{n}$ is a nontrivial fact. In
\cite{RVRevisited} it is proved that $M_{n}$ converges in law to
some $M$, and that $\Law M$ is independent of the choice of the
mollifier $\psi$. Naturally, one expects the stronger result that
$M_{n}$ converges almost surely or at least in probability rather
than just in law. Similarly, the random measure $M$ should be unique
not just in law but as a function of the Gaussian field $X$.

In \cite{DuShKPZ} a related construction with circle averages was
used in the special case where $X$ is the Gaussian free field in
dimension 2. In this very special setting the authors prove almost
sure convergence of $M_{n}$.

\subsection{The new definition}

In this paper we introduce a new definition of GMC that is based on
our view of $M$ as a function of the field $X$ rather than a standalone
random measure.

Recall that a Cameron-Martin vector (or admissible shift) of the Gaussian
field $X$ is a deterministic function $\xi$ on $\mathcal{T}$, such
that the distribution of $X+\xi$ is absolutely continuous with respect
to that of $X$\footnote{The Cameron-Martin space is dual to the Hilbert space of measurable
linear functionals of the field, i.e. ``test functions''. For the
reason explained in Appendix in our setting the space of test functions
contains $L^{2}\left(\mathcal{T},\mu^{\prime}\right)$ for some equivalent
measure $\mu^{\prime}\sim\mu$, so that Cameron-Martin vectors are
representable by $\mu$-equivalence classes of functions.}. We denote by $H$ the space of Cameron-Martin vectors.

Our starting point is the following basic observation: the ``exponential''
behavior \eqref{eq:FormalGMC} of a GMC $M$ can be characterized
by the way $M$ changes when the field $X$ is shifted by Cameron-Martin
vectors. Namely, for all $\xi\in H$ the following should hold almost
surely:
\begin{equation}
M\left(X+\xi,dt\right)=e^{\xi\left(t\right)}M\left(X,dt\right).\label{eq:DefGMC4Field}
\end{equation}
This property is taken as our definition of GMC. A GMC is called \emph{subcritical}
if $\E M$ is $\sigma$-finite, in which case we will often assume
for convenience that $\E M=\mu$. This does not restrict generality,
as will be explained in Remark \ref{rem:EM}. In this paper we only
deal with subcritical GMC theory.

One obvious feature of our definition is that unlike the previous
ones, it is not tied to any particular \emph{construction} of GMC.
On the other hand, for any particular construction that exhibits $M$
as a function of $X$ it is typically easy to check that it satisfies
\eqref{eq:DefGMC4Field}, at least for a dense subspace of Cameron-Martin
shifts, which turns out to be enough. This facilitates the comparison
of different constructions; in particular, the seemingly complicated
problem of independence of the approximation procedure reduces to
the uniqueness problem for our notion of subcritical GMC, which turns
out to be remarkably easy.

In terms of generality, our notion of GMC includes both Kahane's GMC
\cite{KahaneSur} and the subcritical and critical GMC over logarithmic
fields as constructed in \cite{RVRevisited,DuShKPZ,RVReview}, with
the caveat that it retains information about the dependence on the
underlying Gaussian field. On the other hand, it does not include
distributional limits of GMCs that are not measurable with respect
to the Gaussian field even when properly coupled to it, most notably
the atomic supercritical GMC \cite{MRVSupercritical,BJRVKPZ}, nor
does it include the complex GMC \cite{LRVComplex}, which is not a
random positive measure but rather a complex-valued field of almost
surely infinite variation.

\subsection{Approximation}

The main result of the paper is a general approximation theorem for
subcritical GMC. A version of it can be stated as follows.

We assume that $X$ is defined on test functions in $L^{2}\left(\mu\right)$,
so that the Cameron-Martin space $H$ is embedded into $L^{2}\left(\mu\right)$.
Let $A_{n}:H\to L^{2}\left(\mu\right)$ be bounded operators that
converge strongly to the identity embedding $H\to L^{2}\left(\mu\right)$.
We use them to define the approximating fields $X_{n}:=A_{n}X$:
\[
\intop X_{n}\left(t\right)f\left(t\right)\mu\left(dt\right):=\intop X\left(t\right)A_{n}^{\ast}f\left(t\right)\mu\left(dt\right).
\]
Assume that the covariance kernels of the fields $X$ and $X_{n}$
are Hilbert-Schmidt, i.e. there are functions $K_{n},K\in L^{2}\left(\mathcal{T}\times\mathcal{T},\mu\otimes\mu\right)$,
such that for all $f\in L^{2}\left(\mu\right)$ we have 
\[
\E\left(\intop X_{n}\left(t\right)f\left(t\right)\mu\left(dt\right)\right)^{2}=\intop K_{n}\left(t,s\right)f\left(t\right)f\left(s\right)\mu\left(dt\right)\mu\left(ds\right),
\]
and similarly for $X_{n}$ and $K_{n}$ replaced by $X$ and $K$.
Assume that $K_{n}\to K$ in measure ($\mu\otimes\mu$).
\begin{thm}
If there exist subcritical GMCs $M_{n}$ over the fields $X_{n}$
with the same expectation $\mu$, and $\left\{ M_{n}\left[\mathcal{T}\right]\right\} $
are uniformly integrable then there exists a subcritical GMC $M$
over $X$ with expectation $\mu$, and $M_{n}\to M$ in the sense
that for every $f\in L^{1}\left(\mu\right)$
\[
\intop f\left(t\right)M_{n}\left(X_{n},dt\right)\overset{L^{1}\left(\OMEGA,\P\right)}{\to}\intop f\left(t\right)M\left(X,dt\right).
\]
\label{thm:ApproximationA}
\end{thm}
Later we restate and prove the approximation theorem in different
notation as Theorem \ref{thm:Approximation}. The assumptions there
are only marginally more general --- most notably, the approximating
fields are only assumed to be jointly Gaussian, not necessarily measurable
with respect to the limiting field. We also show in Theorem \ref{thm:HilbertSchmidtMoments}
and its Corollary \ref{cor:Moments} that the assumption that the
covariances $K_{n}$ and $K$ are Hilbert-Schmidt, at least with respect
to some equivalent measure $\mu^{\prime}\sim\mu$, follows from the
existence of a subcritical GMC over $X_{n}$, and therefore does not
need to be included as a separate clause.

In spite of the generality of Theorem \ref{thm:Approximation}, even
for logarithmic fields \eqref{eq:LogKernel} and convolution approximations
$A_{n}\xi:=\xi\ast\psi_{1/n}$ as in \eqref{eq:Convolution} our result
is stronger than the approximation theorem of Robert and Vargas \cite{RVRevisited}
in that we assert convergence in probability rather than just in distribution
and identify the limit as a function of the Gaussian field. In this
special case the only condition of Theorem \ref{thm:Approximation}
that is nontrivial to check is the uniform integrability of $M_{n}\left[\mathcal{T}\right]$,
and we will see in Section \ref{sub:ApplicationLog} that it follows
from the results of \cite{KahaneSur}, namely the existence of a GMC
for some specific logarithmic field and Kahane's comparison inequality.

Other approaches to the problems of convergence and independence of
the mollifier for logarithmic GMCs were explored in the papers \cite{JSUniqueness,BerElementary},
both of which appeared after our initial preprint. Unlike our result,
\cite{JSUniqueness} also covers the critical case.

\subsection{Randomized shifts}

A central idea that we employ throughout the paper is to study the
random measure $M$ by considering the measure 
\[
\Q\left(d\omega,dt\right):=\P\left(d\omega\right)M\left(X\left(\omega\right),dt\right)
\]
on $\mathcal{\OMEGA}\times\mathcal{T}$, where $\left(\OMEGA,\mathcal{F},\P\right)$
is the underlying probability space. Since by Cameron-Martin theorem,
multiplying the Gaussian measure by an exponential of a linear functional
amounts to shifting the measure, and the GMC is a generalization of
exponentials, the following fact should come as no surprise.
\begin{thm}
A random measure $M$ is a subcritical GMC over a field $X$ iff for
every positive function $f=f\left(X,t\right)$, measurable with respect
to the field $X$ and $t\in\mathcal{T}$, we have
\begin{equation}
\E\intop f\left(X,t\right)M\left(X,dt\right)=\E\intop f\left(X+K\left(t,\cdot\right),t\right)\mu\left(dt\right),\label{eq:Peyriere}
\end{equation}
where $K$ is the covariance of the field $X$.
\end{thm}
The left-hand side of \eqref{eq:Peyriere} can be written as $\intop f\left(X\left(\omega\right),t\right)\Q\left(d\omega,dt\right)$,
so its right-hand side characterizes the measure $\Q$ on the $\sigma$-algebra
$\sigma\left(X,t\right)$ in terms of the field $X$ and the expectation
$\mu=\E M$. On the other hand, the random measure $M$ can be recovered
from $\Q$ by disintegrating $\Q\left(d\omega,dt\right)$ with respect
to the variable $\omega$, and in fact only the restriction of $\Q$
to $\sigma\left(X,t\right)$ matters, since $M$ is measurable with
respect to $X$. This leads to an important corollary:
\begin{cor}
The subcritical GMC with a given expectation $\mu$ over a given Gaussian
field $X$ is unique whenever it exists.
\end{cor}
In particular, this means that all approximation-based constructions
of subcritical GMC yield the same limit.

There is another point of view on \eqref{eq:Peyriere} that is in
some ways more natural and more appropriate for our purposes, and
which we adopt for the rest of the text. Note that \eqref{eq:Peyriere}
implies that whenever $X$ and $t\in\mathcal{T}$ are sampled independently,
the latter according to $\mu$, the distribution of $X+K\left(t,\cdot\right)$
is absolutely continuous with respect to that of $X$, with density
equal to the total mass $M\left[\mathcal{T}\right]$. We express this
by calling $K\left(t,\cdot\right)$ a \emph{randomized shift} of $X$.
Intuitively, randomized shifts generalize deterministic Cameron-Martin
shifts in the same way as subcritical GMCs generalize exponentials
of Gaussian random variables. More precisely, there is a bijective
correspondence between the two. The relation between subcritical GMCs
and randomized shifts is certainly not a new idea --- for example,
a special case of it is mentioned in \cite{STTranslation} --- but
we have not seen it stated in its proper generality in the literature.
A clean statement of the bijection requires a notational twist, which
we explain next.

Suppose that instead of a Gaussian field we are given an abstract
Gaussian random vector $X$ ``in'' some abstract real separable
Hilbert space $H$, with no a priori relation to the space $\left(\mathcal{T},\mu\right)$.
Then the additional ``field'' structure that is needed in order
to make sense of the definition of \eqref{eq:DefGMC4Field} is a way
to map vectors $\xi\in H$ into ($\mu$-equivalence classes of) functions
on $\mathcal{T}$ --- in other words, a continuous linear operator
$Y:H\to L^{0}\left(\mathcal{T},\mu\right)$, where $L^{0}$ is the
space of $\mu$-equivalence classes of functions equipped with the
topology of convergence in measure. We write $Y$ applied to $\xi\in H$
as $\left\langle Y,\xi\right\rangle \in L^{0}\left(\mathcal{T},\mu\right)$,
and the value of $\left\langle Y,\xi\right\rangle $ at a point $t\in\mathcal{T}$
is written as $\left\langle Y\left(t\right),\xi\right\rangle $. The
operator $Y$ can be viewed as a ``generalized $H$-valued function
on $\mathcal{T}$'', with $\left\langle Y\left(t\right),\xi\right\rangle $
being the ``scalar product'' of the generalized value ``$Y\left(t\right)$''
with $\xi$. Note that $\left\langle Y\left(t\right),\xi\right\rangle $
is only defined for $\mu$-almost all $t$ for every fixed $\xi$,
not for all $\xi$ simultaneously, so ``$Y\left(t\right)$'' may
fail to be a true vector in $H$. In the same way the ``standard
Gaussian'' $X$ itself is not a true random vector in $H$ (unless
it is finite-dimensional), but rather defined as an operator $X:H\to L^{0}\left(\OMEGA,\P\right)$
that takes any $\xi\in H$ to a Gaussian random variable of variance
$\norm{\xi}^{2}$.

We define a generalized random vector in $H$, defined on a probability
space $\left(\OMEGA,\P\right)$ (or $\left(\mathcal{T},\mu\right)$)
as an operator $H\to L^{0}\left(\OMEGA,\P\right)$ (resp. $H\to L^{0}\left(\mathcal{T},\mu\right)$).
For the sake of concreteness, the ``value'' $X\left(\omega\right)$
(resp. $Y\left(t\right)$) of such a vector may be identified with
its sequence of coordinates with respect to a fixed orthonormal basis
$\left\{ e_{n}\right\} \subset H$, i.e. scalar products $\left\langle X\left(\omega\right),e_{n}\right\rangle $
and $\left\langle Y\left(t\right),e_{n}\right\rangle $.

From the point of view described above the objects $X$ and $Y$ are
treated on equal grounds as ``generalized random vectors''. The
defining property of GMC over such a pair $\left(X,Y\right)$ is rewritten
as
\[
M\left(X+\xi,dt\right)=e^{\left\langle Y\left(t\right),\xi\right\rangle }M\left(X,dt\right).
\]
Finally, the relation between GMCs and randomized shifts can be stated
as follows:
\begin{thm}
There exists a subcritical GMC $M$ over $\left(X,Y\right)$ iff $Y$
is a randomized shift, i.e. the distribution of $X+Y\left(t\right)$
when $t$ is sampled independently according to $\mu$ is absolutely
continuous with respect to that of $X$. If the subcritical GMC $M$
does exist then for every $\left(X,t\right)$-measurable function
$f$, we have
\[
\E\intop f\left(X,t\right)M\left(X,dt\right)=\E\intop f\left(X+Y\left(t\right),t\right)\mu\left(dt\right).
\]

\end{thm}
To summarize, we arrive at two points of view on a Gaussian field.
The conventional one is that a field is defined by formal integrals
against test functions, and the other one is that a field is a pair
$\left(X,Y\right)$, where $X$ is a standard Gaussian in some Hilbert
space $H$ and $Y$ is some generalized random vector in $H$ indexed
by $\left(\mathcal{T},\mu\right)$. Intuitively, the relation between
them is that the ``value'' of the Gaussian field at a point $t\in\mathcal{T}$
should be ``$\left\langle X,Y\left(t\right)\right\rangle $''. That
these two points of view are equivalent follows from a nontrivial
result of functional analysis --- a factorization theorem due to Maurey
and Nikishin \cite{Nikishin,Maurey1,Maurey2}, as stated in the Appendix.

The ``$\left(X,Y\right)$'' notation for Gaussian fields clarifies
not only the ``GMC $\leftrightarrow$ randomized shift'' relation
but also the approximation theorem. The sequence of jointly Gaussian
fields corresponds to a sequence of couples $\left(X,Y_{n}\right)$
with the same $X$ representing the underlying Gaussian randomness,
and the convergence condition is $Y_{n}\to Y$ in the strong operator
topology, i.e. $\left\langle Y_{n},\xi\right\rangle \overset{L^{0}}{\to}\left\langle Y,\xi\right\rangle $
for each $\xi\in H$. In other words, it is really the randomized
shift $Y$ that is being approximated rather than the integrals of
the Gaussian fields against any particular test function. In the notation
of Theorem \ref{thm:ApproximationA} this corresponds to requiring
that $A_{n}\overset{s}{\to}A$ rather than $A_{n}^{\ast}\overset{s}{\to}A^{\ast}$
(where $\overset{s}{\rightarrow}$ denotes convergence in the strong
operator topology).

\subsection{Kernel regularity}

It is important in the formulation of our approximation theorem that
the covariance kernels of the fields are representable by functions
up to $\mu\otimes\mu$-equivalence. In GMC theory it is customary
to assume this from the beginning, and we are not aware of any known
results that justify this assumption.

Our Theorem \ref{thm:HilbertSchmidtMoments} and its Corollary \ref{cor:Moments}
provide such a result in the case of subcritical GMC. The statement
that we prove there is that the existence of a subcritical GMC implies
that the covariance kernel of the field is indeed a function (up to
$\mu\otimes\mu$-equivalence), and moreover, this function has polynomial
moments with respect to some equivalent measure $\mu^{\prime}\otimes\mu^{\prime}$:
\begin{equation}
K\in\bigcap_{p}L^{p}\left(\mu^{\prime}\otimes\mu^{\prime}\right).\label{eq:PolyMoment}
\end{equation}

Note that the existence of the \emph{$1$-exponential} moment
\begin{equation}
\intop e^{K\left(t,s\right)}\mu\left(dt\right)\mu\left(ds\right)<\infty\label{eq:ExpMoment}
\end{equation}
is already sufficient for the existence of a subcritical GMC. Indeed,
formally, the $1$-exponential moment of $K$ equals $\E\left[\left(M\left[\mathcal{T}\right]\right)^{2}\right]$,
and \eqref{eq:ExpMoment} ensures that Kahane's martingale approximations
are bounded in $L^{2}$, and therefore uniformly integrable.

In relation to this it is appropriate to mention Kahane's $\frac{1}{2}$-exponential
moment conjecture that states that a subcritical GMC exists iff for
some $\mu^{\prime}\sim\mu$ we have
\begin{equation}
\intop e^{\frac{1}{2}K\left(t,s\right)}\mu^{\prime}\left(dt\right)\mu^{\prime}\left(ds\right)<\infty.\label{eq:ExpMoment2}
\end{equation}
This conjecture turned out to be false --- namely, as Sato and Tamashiro
demonstrated in \cite{STTranslation}, for any $\varepsilon>0$ even
\[
\intop e^{\left(1-\varepsilon\right)K\left(t,s\right)}\mu\left(dt\right)\mu\left(ds\right)<\infty
\]
is not sufficient for the existence of a GMC. To find a correct replacement
for Kahane's condition appears to be an open problem.

\subsection{Organization of the paper}
\begin{itemize}
\item In Section \ref{sec:AbstractNonsense} we introduce relevant notion
of ``generalized Gaussian fields'' and the definition of GMC.
\item In Section \ref{sec:MainResults} we formulate the main results of
the paper --- the GMC $\leftrightarrow$ randomized shift bijection
(Theorem \ref{thm:GMCShifts}) in Section \ref{sub:RandomizedShifts},
the kernel regularity theorem (Theorem \ref{thm:HilbertSchmidtMoments}
and Corollary \ref{cor:Moments}) in Section \ref{sub:KernelRegularity}
and the approximation theorem (Theorem \ref{thm:Approximation}) in
Section \ref{sub:Approximation}. In the remaining Section \ref{sub:ApplicationLog}
we apply the approximation theorem to the convolution approximations
of logarithmic fields (Theorem \ref{thm:ApproximationLog}).
\item In Section \ref{sec:GMCShifts} we prove the GMC $\leftrightarrow$
randomized shift bijection (Theorem \ref{thm:GMCShifts}).
\item In Section \ref{sec:KernelRegularity} we prove the kernel regularity
theorem (Theorem \ref{thm:HilbertSchmidtMoments} and Corollary \ref{cor:Moments}).
The reader only interested in the approximation theorem can safely
skip this, as long as (s)he is willing to assume that both the approximating
fields and the limiting field have Hilbert-Schmidt covariances.
\item In Section \ref{sec:Approximation} we prove our main approximation
theorem (Theorem \ref{thm:Approximation}).
\item In the Appendix we state the Maurey-Nikishin factorization theorem
and explain the relation between the ``$\left(X,Y\right)$'' and
the test function point of view on Gaussian fields.
\end{itemize}

\subsection{Notation and standard assumptions}

We always denote by $H$ a separable infinite-dimensional real Hilbert
space; vectors in $H$ are denoted by $\xi,\eta,\dots$, generalized
random vectors (Definition \ref{def:SVect}) --- by uppercase $X,Y,Z,\dots$.
Among the latter, $X$ is reserved for a standard Gaussian in $H$
(Example \ref{ex:StdGaussian}), defined on a standard probability
space $\left(\OMEGA,\mathcal{F},\P\right)$ (i.e. one isomorphic to
a Polish space equipped with a Borel probability measure). $Y,Z,\dots$
are defined on a standard measurable space $\mathcal{T}$ equipped
with a finite or $\sigma$-finite positive measure $\mu$. This $\mathcal{T}$
serves as the parameter space for generalized Gaussian fields.

Assuming $Y$ is defined on $\mathcal{T}$, a Gaussian multiplicative
chaos over the field $\left(X,Y\right)$ (Definition \ref{def:GMC})
is denoted by $M\left(X,dt\right)$ or, in cases of ambiguity, $M_{Y}\left(X,dt\right)$.

The notation $\Law$ is used for the distribution of a generalized
random vector, i.e. the joint distribution of linear functionals of
it. Modifiers like $\Law_{\mu},\Law_{\Q}$ are used when the underlying
probability space is equipped with probability measures $\mu,\Q$.
Similarly, $\E_{\mu},\E_{\Q}$ are used for the expectation with respect
to $\mu,\Q$.

\section{The setup \label{sec:AbstractNonsense}}

\subsection{Generalized Gaussian fields}

For any standard probability space $\left(\mathcal{T},\mu\right)$
we denote by $L^{0}\left(\mathcal{T},\mu\right)$ the space of $\mu$-equivalence
classes of functions on $\mathcal{T}$, equipped with the topology
of convergence in measure.

Throughout the text we fix a real separable Hilbert space $H$.
\begin{defn}
A \emph{generalized $H$-valued function} defined on $\left(\mathcal{T},\mu\right)$
is a continuous linear operator $Y:H\to L^{0}\left(\mathcal{T},\mu\right)$.
Generalized $H$-valued functions defined on the distinguished probability
space $\left(\OMEGA,\P\right)$ are also called \emph{generalized
random vectors} (in $H$). \label{def:SVect}\end{defn}
\begin{rem}
The notion of generalized $H$-valued function only depends on the
equivalence class of measures. Indeed, if $\mu^{\prime},\mu^{\prime\prime}$
are two equivalent probability measures then the topologies of convergence
in measure are the same for $\mu^{\prime}$ and $\mu^{\prime\prime}$,
so $L^{0}\left(\mu^{\prime}\right)$ and $L^{0}\left(\mu^{\prime\prime}\right)$
can be identified in a canonical way. This also allows us to define
$L^{0}$ over a (nonzero) $\sigma$-finite measure as $L^{0}$ over
any equivalent probability measure. For the sake of having the right
definitions in the trivial case, for $\mu=0$ we set $L^{0}\left(\mu\right):=\left\{ 0\right\} $.
\label{rem:L0}
\end{rem}
We use the ``scalar product'' notation $\left\langle Y,\xi\right\rangle $
for a generalized $H$-valued function $Y$ and a vector $\xi\in H$
to denote the corresponding (equivalence class of) function. By an
abuse of notation we will write the value $\left\langle Y,\xi\right\rangle \left(t\right)$
as $\left\langle Y\left(t\right),\xi\right\rangle $.

For the sake of concreteness we may treat any generalized $H$-valued
function $Y$ as a ``true'' equivalence class of a function with
values in $\RR^{\infty}$ by taking its coordinates in a fixed orthonormal
basis $\left(e_{n}\right)$ of $H$. Namely, the ``value $Y\left(t\right)$''
is identified with its sequence of coordinates $Y^{\left(n\right)}\left(t\right):=\left\langle Y\left(t\right),e_{n}\right\rangle $,
which is well-defined up to equality almost everywhere. Not all $\RR^{\infty}$-valued
functions can serve as generalized $H$-valued functions. The necessary
and sufficient condition for a sequence $Y^{\left(n\right)}$ to come
this way from a generalized $H$-valued function is that for all $\xi\in\ell^{2}$
the series $\sum_{n}\xi_{n}Y^{\left(n\right)}$ should converge in
$L^{0}$, so that it is possible to define $\left\langle Y,\xi\right\rangle $
almost everywhere as a measurable linear functional.

The choice of $\RR^{\infty}\supset H$ as the space to host the values
of all generalized $H$-valued functions is highly arbitrary and actually
irrelevant for our purposes. It is, however, important for $Y\left(t\right)$
to have \emph{some} Frechet (thus: standard Borel) space to live in.

We identify true $H$-valued functions with a special case of generalized
$H$-valued functions. Accordingly, ``$Y\left(t\right)\in H$ for
almost all $t$'' means that $Y$ is a true $H$-valued function.
\begin{example}
We call a generalized random vector $X$ \emph{standard Gaussian}
in $H$ if all $\left\langle X,\xi\right\rangle ,\xi\in H$ are centered
Gaussian random variables with variance $\E\left\langle X,\xi\right\rangle ^{2}=\norm{\xi}^{2}$.
\label{ex:StdGaussian}
\end{example}
We say that a generalized $H$-valued function $Y$ on a probability
$\left(\mathcal{T},\mu\right)$ has a weak first moment (with respect
to the measure $\mu$) if for any $\xi\in H$ the function $\left\langle Y,\xi\right\rangle $
is in $L^{1}\left(\mu\right)$. In this case, by a standard application
of the closed graph theorem, there exists a vector $\intop Y\left(t\right)\mu\left(dt\right)\in H$
defined in the obvious way: 
\[
\left\langle \intop Y\left(t\right)\mu\left(dt\right),\xi\right\rangle :=\intop\left\langle Y\left(t\right),\xi\right\rangle \mu\left(dt\right).
\]
We may replace the symbol $\intop\dots\mu\left(dt\right)$ by $\E_{\mu}$,
and $\intop\dots\P\left(d\omega\right)$ by $\E$.
\begin{defn}
A \emph{generalized Gaussian field} on a standard measure space $\left(\mathcal{T},\mu\right)$
is a couple $\left(X,Y\right)$, where $X$ is the standard Gaussian
random vector in $H$ (defined on $\left(\OMEGA,\P\right)$) and $Y$
is a generalized $H$-valued function defined on $\left(\mathcal{T},\mu\right)$.
\label{def:GaussianField}
\end{defn}
As mentioned in the introduction, this point of view on generalized
Gaussian fields is equivalent to the more conventional one in terms
of integration against $L^{2}$ test functions. The equivalence between
the two is nontrivial and involves a factorization theorem due to
Maurey and Nikishin (Theorem \ref{thm:Factorization}). The translation
in both directions is explained in the Appendix.

\subsection{The definition of a GMC}

By a random measure on a measure space $\left(\mathcal{T},\mu\right)$
we always mean a random \emph{positive finite} measure $M$, such
that $\E M$ is $\mu$-absolutely continuous (notation: $\E M\ll\mu$).
The measure $\E M$ is defined by $\left(\E M\right)\left[A\right]:=\E\left(M\left[A\right]\right)$
for all measurable subsets $A\subset\mathcal{T}$. Note that $\E M$
need not be a $\sigma$-finite measure; however, it is equivalent
to a finite one --- namely, $\E\left[\left(M\left[\mathcal{T}\right]\vee1\right)^{-1}M\right]$
--- in the sense that their classes of null sets coincide. Similarly,
the condition $\E M\ll\mu$ is understood in the sense that for any
measurable $A\subset\mathcal{T}$, such that $\mu\left[A\right]=0$,
we have $\E M\left[A\right]=0$, or equivalently, $M\left[A\right]=0$
a.s.

Note that even though $\E M$ is $\mu$-absolutely continuous, $M$
itself may be almost surely $\mu$-singular.

We refer to the vectors $\xi\in H$ as Cameron-Martin shifts (of the
Gaussian $X$). For the necessary background on them the reader is
referred to \cite[Theorem 14.1]{Janson}.
\begin{defn}
A random measure $M$ on $\left(\mathcal{T},\mu\right)$ is called
a \emph{Gaussian multiplicative chaos} (GMC) over the Gaussian field
$\left(X,Y\right)$ if
\begin{enumerate}
\item $\E M\ll\mu$
\item $M$ is measurable with respect to $X$ (which allows us to write
$M=M\left(X\right)$);
\item For all vectors $\xi\in H$
\begin{equation}
M\left(X+\xi,dt\right)=e^{\left\langle Y\left(t\right),\xi\right\rangle }M\left(X,dt\right)\text{ a.s.}\label{eq:DefGMC}
\end{equation}

\end{enumerate}
The GMC is called \emph{subcritical} if $\E M$ is $\sigma$-finite.
\label{def:GMC}
\end{defn}
Instead of ``GMC over the Gaussian field $\left(X,Y\right)$'' we
may also say ``GMC associated to $Y$'' with the Gaussian $X$ understood
implicitly.

Formula \eqref{eq:DefGMC} requires a couple of comments. First, $X+\xi$
is a shifted Gaussian, so by the Cameron-Martin theorem, its distribution
is equivalent to that of $X$, which makes $M\left(X+\xi\right)$
well-defined. Second, even though $\left\langle Y,\xi\right\rangle $
is only defined almost everywhere with respect to $\mu$, and in the
interesting cases $M$ is almost surely $\mu$-singular, $e^{\left\langle Y,\xi\right\rangle }M$
is still well-defined, precisely because $\E M$ is $\mu$-absolutely
continuous. Indeed, if $\varphi,\tilde{\varphi}$ are two measurable
functions on $\mathcal{T}$ that are equal $\mu$-almost everywhere
to $e^{\left\langle Y,\xi\right\rangle }$ then for $\P$-almost all
$\omega$ and $M\left(X\left(\omega\right),dt\right)$-almost all
$t$ we have $\varphi\left(t\right)=\tilde{\varphi}\left(t\right)$,
thus almost surely $\varphi M=\tilde{\varphi}M$.
\begin{example}
If $Y\left(t\right)\in H$ for almost all $t$ then
\begin{equation}
M\left(X,dt\right):=\exp\left[\left\langle X,Y\left(t\right)\right\rangle -\frac{1}{2}\norm{Y\left(t\right)}^{2}\right]\mu\left(dt\right)\label{eq:TrivialGMC}
\end{equation}
is a subcritical GMC over the Gaussian field $\left(X,Y\right)$ with
expectation $\E M=\mu$. \label{ex:TrivialGMC}\end{example}
\begin{rem}
Obviously, in the definition of GMC only the equivalence class of
$\mu$ matters. However, in the subcritical case we will sometimes
assume that $\E M=\mu$, which is no loss of generality because a
GMC on $\left(\mathcal{T},\mu\right)$ is also a GMC on $\left(\mathcal{T},\E M\right)$.
Furthermore, another common assumption will be that $\E M=\mu$ is
finite, which is also no loss of generality because our theory is
essentially local, i.e. the statements for the whole space $\left(\mathcal{T},\mu\right)$
reduce to those for its subsets of finite measure. \label{rem:EM}
\end{rem}
Finally we would like to remark there exist GMCs that are not subcritical.
So far the only examples known to the author are the critical GMCs
over logarithmic fields and their hierarchical counterparts \cite{HuShiMinimal,DSRVCriticalConvergence,DSRVRenormalization}.
For these critical GMCs we have for any measurable set $A\subset\mathcal{T}$
\[
\E M\left[A\right]=\begin{cases}
0, & \mu\left[A\right]=0\\
\infty, & \mu\left[A\right]>0
\end{cases}
\]
so that it is impossible to normalize them by expectation. In these
cases $\E M$ is a non-$\sigma$-finite measure that we agree to call
$\mu$-absolutely continuous; its density with respect to $\mu$ is
almost everywhere infinite.

\section{Main results \label{sec:MainResults}}

\subsection{Randomized shifts \label{sub:RandomizedShifts}}
\begin{defn}
A generalized $H$-valued function $Y$ defined on $\left(\mathcal{T},\mu\right)$
is called a \emph{randomized shift} if
\[
\Law_{\P\otimes\mu}\left[X+Y\right]\ll\Law_{\P}X.
\]

\end{defn}
Note that being a randomized shift only depends on the equivalence
class of $\Law_{\mu}Y$.
\begin{example}
[Trivial shifts] By the Cameron-Martin theorem every $H$-valued
function $Y:\mathcal{T}\to H$ (not ``generalized''!) is a randomized
shift.
\end{example}
These Cameron-Martin shifts are viewed as ``trivial''. There are
less trivial ones:
\begin{example}
[Gaussian shifts] By the Hajek-Feldman theorem \cite[Theorem 6.3.2]{Bog},
a Gaussian generalized random vector in $H$ (i.e. that for which
all linear functionals are Gaussian) is a randomized shift iff its
covariance is Hilbert-Schmidt. Note also that being ``trivial''
in the above sense is equivalent to the covariance being trace class.
\end{example}
In Theorem \ref{thm:GMCShifts} we describe the relation between subcritical
GMC over Gaussian fields with parameter space $\left(\mathcal{T},\mu\right)$
and randomized shifts defined on $\left(\mathcal{T},\mu\right)$.
Here it is convenient to view $\left(\mathcal{T},\mu\right)$ as an
additional source of randomness, so functions on $\OMEGA\times\mathcal{T}$
are treated as random variables --- in particular, the projection
map $\OMEGA\times\mathcal{T}\to\mathcal{T}$ is treated as ``the''
random point $t$ in $\mathcal{T}$. Accordingly, we use the notation
$\Law_{\mu},\Law_{\P\otimes\mu},\dots$ for the law of a generalized
random vector defined on the probability space $\left(\mathcal{T},\mu\right),\left(\OMEGA\times\mathcal{T},\P\otimes\mu\right),\dots$.
\begin{thm}
There exists a subcritical GMC $M$ over the Gaussian field $\left(X,Y\right)$
with expectation $\mu$ iff $Y$ is a randomized shift, in which case
under the Peyrière measure on $\OMEGA\times\mathcal{T}$
\begin{equation}
\Q\left(d\omega,dt\right):=\P\left(d\omega\right)M\left(X\left(\omega\right),dt\right)\label{eq:DefQ}
\end{equation}
we have
\begin{equation}
\Law_{\Q}\left[X,t\right]=\Law_{\P\otimes\mu}\left[X+Y\left(t\right),t\right].\label{eq:QShift}
\end{equation}
\label{thm:GMCShifts}
\end{thm}
Note that \eqref{eq:QShift} characterizes uniquely the measure $\Q$
on the $\sigma$-algebra generated by $\left(X,t\right)$, so by disintegration
with respect to $X$ it also characterizes $M\left(X\right)$. Therefore,
we have the following important corollary: a subcritical GMC is unique.
\begin{cor}
A subcritical GMC with a given expectation $\mu$, associated to a
given $Y$, is unique, whenever it exists. $M$ can be recovered from
$\mu$ and $Y$ as follows:
\[
M\left[\mathcal{T}\right]=\Law_{\P\otimes\mu}\left[X+Y\left(t\right)\right]/\Law X,
\]
\[
\left(M\left[\mathcal{T}\right]\right)^{-1}M\left(x,dt\right)=\Lawcm{\P\otimes\mu}{t\in dt}{X+Y\left(t\right)=x},
\]
where $\dots/\Law X$ denotes the Radon-Nikodym density viewed as
a function of $X$, and $\Lawc{\cdot}{\cdot}$ denotes the conditional
distribution. \label{cor:Uniqueness}
\end{cor}
Note that unlike in the previous approaches to GMC theory, we have
proven uniqueness of $M$ as a function of $X$ rather than just in
law.

It is instructive to note the role of the subcriticality of $M$ in
Theorem \ref{thm:GMCShifts}. For every GMC, not necessarily a subcritical
one, we can construct the measure $\Q$ as in \eqref{eq:DefQ}. In
general, $\Q$ is only $\sigma$-finite, and subcriticality is equivalent
to the finiteness of $\mu$-almost all fiber measures $\Q_{t}$ in
the disintegration
\[
\Q\left(d\omega,dt\right)=\Q_{t}\left(d\omega\right)\mu\left(dt\right).
\]
Indeed, the density of $\E M$ with respect to $\mu$ at $t\in\mathcal{T}$
is equal to the total mass of the fiber measure $\Q_{t}$. The proof
of Theorem \ref{thm:GMCShifts} proceeds essentially by verifying
that the fiber $\Q_{t}$ behaves under shifts by $\xi\in H$ like
the standard Gaussian measure shifted by $Y\left(t\right)$ in the
sense that it satisfies the corresponding Cameron-Martin formula.

Note that among \emph{finite} measures there is a unique one (up to
scaling) that satisfies the Cameron-Martin formula, namely the Gaussian
itself, so that the behavior of the GMC under shifts completely characterizes
$\Q$. However, the uniqueness argument fails without the subcriticality
assumption because there are many different $\sigma$-finite measures
that satisfy the Cameron-Martin formula. For example, for any positive
sequence $\left(C_{n}\right)$, such that $\sum_{n}e^{-C_{n}^{2}}<\infty$
there is a unique $\sigma$-finite measure $\gamma$ on $\RR^{\infty}$
that satisfies the Cameron-Martin formula for all shifts in $H:=\ell^{2}$
and whose restriction to $\prod_{n}\left[-C_{n},C_{n}\right]$ is
$\gamma^{\op{res}}:=\bigotimes_{n}\left(Z_{n}^{-1}\I\left\{ \left|x_{n}\right|\le C_{n}\right\} e^{-x_{n}^{2}/2}dx_{n}\right)$,
$Z_{n}:=\intop_{-C_{n}}^{C_{n}}e^{-x^{2}/2}dx$. One can construct
such a measure by gluing together the ``compatible'' measures $e^{-\left\langle \cdot,\xi\right\rangle -\frac{1}{2}\norm{\xi}^{2}}\left(S_{\xi}\right)_{\ast}\gamma^{\op{res}}$
(where $S_{\xi}$ is the shift $x\mapsto x+\xi$). The resulting measure
$\gamma$ is singular to the Gaussian iff $\sum_{n}C_{n}^{-1}e^{-C_{n}^{2}/2}=\infty$.
In fact, by choosing different sequences $\left(C_{n}\right)$ one
can produce a continuum of mutually singular $\sigma$-finite measures
that satisfy the Cameron-Martin formula. Due to this difficulty the
uniqueness problem for GMCs without the subcriticality assumption
remains open.

\subsection{Regularity of the kernel \label{sub:KernelRegularity}}

A priori for a Gaussian field $\left(X,Y\right)$ one can define the
covariance kernel, ``$K\left(t,s\right)=\left\langle Y\left(t\right),Y\left(s\right)\right\rangle $'',
as the formal kernel of the bilinear form on the $L^{2}\left(\mu^{\prime}\right)$
test functions (see the Appendix). Namely,
\[
\begin{aligned}\intop K\left(t,s\right)f\left(t\right)f\left(s\right)\mu^{\prime}\left(dt\right)\mu^{\prime}\left(ds\right) & :=\norm{\intop f\left(t\right)Y\left(t\right)\mu^{\prime}\left(dt\right)}^{2}\\
 & =\E\left\langle X,\intop f\left(t\right)Y\left(t\right)\mu^{\prime}\left(dt\right)\right\rangle ^{2}.
\end{aligned}
\]
Not all bounded operators in $L^{2}$ are integral operators, so neither
are such $K$'s represented by ``true'' functions on $\mathcal{T}\times\mathcal{T}$.
However, we will prove that in the GMC theory this pathology does
not happen, i.e. whenever a Gaussian field admits a subcritical GMC,
the covariance kernel $K$ is actually representable by a function
on $\mathcal{T}\times\mathcal{T}$. Furthermore, this function has
all moments with respect to some equivalent measure $\mu^{\prime}\otimes\mu^{\prime}$.
The proof proceeds by translating this property to an equivalent statement
about the randomized shift $Y$ and relies on the factorization theorem
(Theorem \ref{thm:Factorization}) for the construction of $\mu^{\prime}$.

For a Hilbert space $H$ we denote by $H^{\otimes n}$ its Hilbert
$n$-th tensor power, also called the space of Hilbert-Schmidt tensors.
By definition, it is the completion of the algebraic tensor power
$H_{\op{alg}}^{\otimes n}$ with respect to the scalar product
\[
\left\langle \xi_{1}\otimes\dots\otimes\xi_{n},\eta_{1}\otimes\dots\otimes\eta_{n}\right\rangle :=\left\langle \xi_{1},\eta_{1}\right\rangle \dots\left\langle \xi_{n},\eta_{n}\right\rangle ,
\]
extended from decomposable tensors to all tensors by multilinearity.
We denote by $\norm{\cdot}_{2}$ the Hilbert space norm corresponding
to this scalar product.

For a standard Gaussian $X$ in $H$ there is a well-known theory
of random variables that are polynomial in $X$, also known as the
Wick calculus (see \cite[Chapter III]{Janson}). Its basic construction
is the Wick product of jointly Gaussian random variables, denoted
by $\Wick{\left\langle X,\xi_{1}\right\rangle \dots\left\langle X,\xi_{n}\right\rangle }$,
and defined by the polarization of the identity
\[
\Wick{\left\langle X,\xi\right\rangle \dots\left\langle X,\xi\right\rangle }=\norm{\xi}^{n}h_{n}\left(\norm{\xi}^{-1}\left\langle X,\xi\right\rangle \right),
\]
where $h_{n}$ is the $n$-th Hermite polynomial
\[
h_{n}\left(x\right):=e^{-\frac{1}{2}\frac{\partial^{2}}{\partial x^{2}}}x^{n}=x^{n}-\frac{1}{2}n\left(n-1\right)x^{n-2}+\dots
\]
The basic fact is that $\left\langle \Wick{X^{\otimes n}},\lambda\right\rangle $,
defined initially for finite-rank tensors (i.e. elements of the algebraic
tensor product) $\lambda\in H_{\op{alg}}^{\otimes n}$ by 
\[
\left\langle \Wick{X^{\otimes n}},\xi_{1}\otimes\dots\otimes\xi_{n}\right\rangle :=\Wick{\left\langle X,\xi_{1}\right\rangle \dots\left\langle X,\xi_{n}\right\rangle },
\]
extends by $L^{2}$-continuity to all Hilbert-Schmidt tensors $\lambda\in H^{\otimes n}$.
In our language this is stated as follows:
\[
\Wick{X^{\otimes n}}\text{ is a generalized random vector in }H^{\otimes n}.
\]
It turns out that this implies a corresponding property for a randomized
shift $Y$ --- with the crucial difference that $Y$, unlike $X$,
does not need the Wick renormalization:
\[
Y^{\otimes n}\text{ is a generalized random vector in }H^{\otimes n}.
\]
This is the content of our Theorem \ref{thm:HilbertSchmidtMoments},
and it turns out to be equivalent to the existence of the $n$-th
moment of the covariance kernel $K$ with respect to $\mu^{\prime}\otimes\mu^{\prime}$
for some equivalent measure $\mu^{\prime}\sim\mu$.
\begin{thm}
Let $X$ be a standard Gaussian in $H$, and let $\left(\mathcal{T},\mu\right)$
be a standard probability space. Let $Y$ be a randomized shift, defined
on $\left(\mathcal{T},\mu\right)$. Then for every $n\in\NN$ there
is an equivalent measure $\mu_{n}^{\prime}\sim\mu$ on $\mathcal{T}$,
such that under $\mu_{n}^{\prime}$ all $\left\langle Y,\xi\right\rangle $
have finite absolute $n$-th moment, and the symmetric tensor $\E_{\mu_{n}^{\prime}}Y^{\otimes n}$,
defined by the polarization of $\xi\mapsto\E_{\mu_{n}^{\prime}}\left\langle Y,\xi\right\rangle ^{n}$,
is Hilbert-Schmidt. \label{thm:HilbertSchmidtMoments}\end{thm}
\begin{cor}
In the setting of Theorem \ref{thm:HilbertSchmidtMoments} the quadratic
form
\[
f\mapsto\norm{\intop f\left(t\right)Y\left(t\right)\mu_{n}^{\prime}\left(dt\right)}^{2}
\]
is Hilbert-Schmidt. Thus there exists a unique symmetric function
$K\in L^{2}\left(\mu_{n}^{\prime}\otimes\mu_{n}^{\prime}\right)$,
such that
\[
\norm{\intop f\left(t\right)Y\left(t\right)\mu_{n}^{\prime}\left(dt\right)}^{2}=\intop K\left(t,s\right)f\left(t\right)f\left(s\right)\mu_{n}^{\prime}\left(dt\right)\mu_{n}^{\prime}\left(ds\right)
\]
for all $f\in L^{2}\left(\mu_{n}^{\prime}\right)$. Moreover,
\[
\intop\left(K\left(t,s\right)\right)^{n}\mu_{n}^{\prime}\left(dt\right)\mu_{n}^{\prime}\left(ds\right)<\infty.
\]
\label{cor:Moments}\end{cor}
\begin{rem}
Yet another interpretation of this result could be: the existence
of a subcritical GMC, which is the Wick exponential of the Gaussian
field, implies the existence of all Wick powers ``$\Wick{\left(X\left(t\right)\right)^{n}}$''
of that field, due to the formal identity
\begin{equation}
\E\left|\intop\Wick{\left(X\left(t\right)\right)^{n}}\mu_{n}^{\prime}\left(dt\right)\right|^{2}=\intop\left(K\left(t,s\right)\right)^{n}\mu_{n}^{\prime}\left(dt\right)\mu_{n}^{\prime}\left(ds\right)<\infty.\label{eq:WickL2}
\end{equation}

\end{rem}

\begin{rem}
For a randomized shift $Y$ there is a single measure $\mu^{\prime}\sim\mu$,
such that
\[
\intop\left(K\left(t,s\right)\right)^{n}\mu^{\prime}\left(dt\right)\mu^{\prime}\left(ds\right)<\infty
\]
for all $n$ simultaneously. Such a measure can be constructed as
follows:
\[
\mu^{\prime}\left(dt\right):=\inf_{n}\left(r_{n}\mu_{n}^{\prime}\left(dt\right)/\mu\left(dt\right)\right)\cdot\mu\left(dt\right),
\]
where $\left(r_{n}\right)$ is a sequence of positive numbers that
decreases fast enough (namely, so that $\mu\left(dt\right)/\mu_{n}^{\prime}\left(dt\right)=O\left(r_{n}\right),n\to\infty$
for $\mu$-almost all $t$). This fact is not used anywhere in the
text.
\end{rem}

\begin{rem}
Instead of the factorization theorem, which only implies that $K^{n}$
is a bounded bilinear form on $L^{2}\left(\mu^{\prime}\right)$ test
functions for some $\mu^{\prime}$, one can use a sharp bound, replacing
$L^{2}$ by the Orlicz space $L\left(\log L\right)^{n/2}$. This follows
by Orlicz space duality from the a priori $\exp\left(-\const\cdot x^{2/n}\right)$
tail decay of the distribution of random variables of the form $\left\langle \lambda,\Wick{X^{\otimes n}}\right\rangle $,
$\lambda\in H^{\otimes n}$ (see e.g. \cite[Theorem 6.7]{Janson}),
which carries over with a change of measure to $\left\langle \lambda,\Wick{\left(X+cY\right)^{\otimes n}}\right\rangle $
and thus also to $\left\langle \lambda,Y^{\otimes n}\right\rangle $.
The resulting bound,
\[
\intop\left(K\left(t,s\right)\right)^{n}f\left(t\right)g\left(s\right)\mu^{\prime}\left(dt\right)\mu^{\prime}\left(ds\right)\le\const\cdot\norm f_{L\left(\log L\right)^{2/n}}\norm g_{L\left(\log L\right)^{n/2}},
\]
yields the following estimate: if $\mathcal{T}=\left[0,1\right],\mu^{\prime}=\op{Lebesgue}$,
$f:=\frac{1}{\varepsilon}\I\left[a,a+\varepsilon\right],g:=\frac{1}{\varepsilon}\I\left[b,b+\varepsilon\right]$
for some $a,b\in\left[0,1-\varepsilon\right]$, then we have 
\[
\op{ess}\sup\left(K^{n}\ast\left(\varepsilon^{-1}\I\left[0,\varepsilon\right]\right)^{\otimes2}\right)=O\left(\left|\log\varepsilon\right|^{n}\right),\varepsilon\to0.
\]
By considering the logarithmic kernels one can see that this bound
is sharp. This will not be used in the paper.
\end{rem}

\begin{rem}
A weaker bound, $K\in L^{2}\left(\mu^{\prime}\otimes\mu^{\prime}\right)$
for some $\mu^{\prime}\sim\mu$, can be proved without the subcriticality
assumption using a different approach. This will be presented elsewhere.
\end{rem}

\subsection{Approximation \label{sub:Approximation}}

Let $Y_{n},n\ge1$ be randomized shifts defined on a probability space
$\left(\mathcal{T},\mu\right)$. Let $K_{Y_{n}Y_{n}}\left(t,s\right):=\left\langle Y_{n}\left(t\right),Y_{n}\left(s\right)\right\rangle $
be the corresponding kernel, which, by Corollary \ref{cor:Moments},
is well-defined as a function on $\left(\mathcal{T}\times\mathcal{T},\mu\otimes\mu\right)$.
Let $M_{Y_{n}}$ be the subcritical GMC associated to $Y_{n}$ with
expectation $\mu$.

Our main result on the approximation of subcritical GMC is as follows:
\begin{thm}
Assume that:
\begin{itemize}
\item The family of random variables $\left\{ M_{Y_{n}}\left[\mathcal{T}\right]\right\} $
is uniformly integrable;
\item There exists a generalized $H$-valued function $Y$ defined on $\left(\mathcal{T},\mu\right)$
that is the limit of $Y_{n}$ in the sense that
\begin{equation}
\forall\xi\in H:\left\langle Y_{n},\xi\right\rangle \overset{L^{0}\left(\mu\right)}{\to}\left\langle Y,\xi\right\rangle .\label{eq:ConvergenceOfShifts}
\end{equation}

\end{itemize}
Then $Y$ is a randomized shift. If, furthermore,
\begin{itemize}
\item the kernels $K_{Y_{n}Y_{n}}$ converge to $K_{YY}$ in $L^{0}\left(\mu\otimes\mu\right)$,
\end{itemize}
then the subcritical GMC $M_{Y}$ (associated to $Y$ with expectation
$\mu$) is the limit of $M_{Y_{n}}$ in the sense that
\begin{equation}
\forall f\in L^{1}\left(\mu\right):\intop f\left(t\right)M_{Y_{n}}\left(X,dt\right)\overset{L^{1}}{\to}\intop f\left(t\right)M_{Y}\left(X,dt\right).\label{eq:ConvergenceOfGMC}
\end{equation}

\label{thm:Approximation}
\end{thm}

\subsection{Application to logarithmic kernels \label{sub:ApplicationLog}}

Let $\mathcal{T}\subset\RR^{d}$ be a bounded domain, let $\mu$ be
the Lebesgue measure on $\mathcal{T}$, and let $K$ be a positive
definite Hilbert-Schmidt kernel on $\left(\mathcal{T},\mu\right)\times\left(\mathcal{T},\mu\right)$,
such that for some $\delta>0$
\begin{equation}
K_{YY}\left(t,s\right)\le\left(2d-\delta\right)\log\norm{t-s}^{-1}+O\left(1\right),\forall t,s\in\mathcal{T}.\label{eq:BoundOnK}
\end{equation}
Consider also a bounded function $\psi$ on $\RR^{d}$ with compact
support, such that $\psi\ge0$, $\intop\psi\left(x\right)dx=1$, and
denote $\psi_{\varepsilon}\left(x\right):=\varepsilon^{-d}\psi\left(\varepsilon^{-1}x\right)$.

Take any generalized random vector $Y$ with $\left\langle Y\left(t\right),Y\left(s\right)\right\rangle :=K_{YY}\left(t,s\right)$.
In order to construct one we may start with a Gaussian field on $\left(\mathcal{T},\mu\right)$
defined by its integrals against test functions. Namely, to every
test function $f\in L^{2}\left(\mu\right)$ we associate the Gaussian
variable $\left\langle X,Af\right\rangle $, where $A:L^{2}\left(\mu\right)\to H$
is a bounded linear operator, such that 
\[
\left\langle Af,Ag\right\rangle =\intop K_{YY}\left(t,s\right)f\left(t\right)g\left(s\right)\mu\left(dt\right)\mu\left(ds\right),
\]
and we construct $Y$ as the composition $H\overset{A}{\to}L^{2}\left(\mathcal{T},\mu\right)\overset{\id}{\to}L^{0}\left(\mathcal{T},\mu\right)$.
\begin{thm}
Let 
\[
Y_{\varepsilon}\left(t\right):=\intop_{\mathcal{T}}Y\left(t^{\prime}\right)\psi_{\varepsilon}\left(t-t^{\prime}\right)dt^{\prime},
\]
\[
K_{Y_{\varepsilon},Y_{\varepsilon}}\left(t,s\right):=\left\langle Y_{\varepsilon}\left(t\right),Y_{\varepsilon}\left(s\right)\right\rangle =\intop_{\mathcal{T}\times\mathcal{T}}K_{YY}\left(t^{\prime},s^{\prime}\right)\psi_{\varepsilon}\left(t-t^{\prime}\right)\psi_{\varepsilon}\left(s-s^{\prime}\right)dt^{\prime}\,ds^{\prime},
\]
\[
M_{Y_{\varepsilon}}\left(dt\right):=\exp\left[\left\langle X,Y_{\varepsilon}\left(t\right)\right\rangle -\frac{1}{2}K_{Y_{\varepsilon}Y_{\varepsilon}}\left(t,t\right)\right]dt.
\]
Then there exists a subcritical GMC $M_{Y}$ over $\left(X,Y\right)$,
and $M_{Y_{\varepsilon}}\to M_{Y}$ in probability (the space of measures
is equipped with the weak topology). This $M_{Y}$ does not depend
on the function $\psi$ used for approximation. \label{thm:ApproximationLog}
\end{thm}
Note that in this case $K_{Y_{\varepsilon}Y_{\varepsilon}}$ is a
continuous kernel, so the Gaussian field $\left(\left\langle X,Y_{\varepsilon}\left(t\right)\right\rangle \right)_{t\in\mathcal{T}}$
has well-defined values at each $t\in\mathcal{T}$. The field $\left\langle X,Y_{\varepsilon}\left(\cdot\right)\right\rangle $
is obtained by convolving our generalized field $\left\langle X,Y\left(\cdot\right)\right\rangle $
with $\psi_{\varepsilon}$, which makes sense, since $\psi_{\varepsilon}$
is allowed as a test function.

We will see that Theorem \ref{thm:ApproximationLog} follows from
Theorem \ref{thm:Approximation} once we have a way of verifying the
uniform integrability assumption of the latter. This is done using
a known result on existence of GMC for specific logarithmic kernels
and Kahane's comparison inequality \cite{KahaneSur}.
\begin{thm}
[See \cite{KahaneSur}] Consider the following kernels on $\mathcal{T}\times\mathcal{T}$:
\[
\tilde{K}_{C,\gamma}\left(t,s\right):=\gamma^{2}\intop_{1}^{C}e^{-u\norm{t-s}}\frac{du}{u}=\gamma^{2}\log\left(C\wedge\norm{t-s}^{-1}\right)+O\left(1\right).
\]
Then for $\gamma<\sqrt{2d}$ the family of GMCs with these kernels
is uniformly integrable. \label{thm:ExistenceClassical}
\end{thm}

\begin{thm}
[See \cite{KahaneSur}] Let $\left(X,Y_{1}\right)$ and $\left(X,Y_{2}\right)$
be Gaussian fields on $\mathcal{T}$ with continuous (or, more generally,
trace class) covariance kernels $K_{1},K_{2}$. Assume that
\[
\forall t,s:K_{1}\left(t,s\right)\le K_{2}\left(t,s\right).
\]
Then for every convex function $f:\RR_{+}\to\RR_{+}$
\[
\E f\left(\intop\exp\left[\left\langle X,Y_{1}\left(t\right)\right\rangle -\frac{1}{2}K_{1}\left(t,t\right)\right]dt\right)\le\E f\left(\intop\exp\left[\left\langle X,Y_{2}\left(t\right)\right\rangle -\frac{1}{2}K_{2}\left(t,t\right)\right]dt\right).
\]
 \label{thm:KahaneInequality}\end{thm}
\begin{proof}
[Proof of Theorem \ref{thm:ApproximationLog}] All assumptions of
Theorem \ref{thm:Approximation} except for the uniform integrability
are quite trivial to check.

The convergence assumption $Y_{\varepsilon}\to Y$ amounts to the
following. For every $\xi\in H$ we associate the function $\xi\left(t\right):=\left\langle \xi,Y\left(t\right)\right\rangle $
which belongs to the Cameron-Martin space of the field (equivalently,
the reproducing kernel Hilbert space associated to $K$). The assumption
$Y_{\varepsilon}\to Y$ is equivalent to 
\[
\xi\left(\cdot\right)\ast\psi_{\varepsilon}\to\xi\left(\cdot\right)
\]
for every such function, which is trivial, since $\xi\left(\cdot\right)\in L^{2}$.
The condition $K_{\varepsilon}\to K$ in $L^{0}\left(\mu\otimes\mu\right)$
is also trivially satisfied.

To verify uniform integrability we use Kahane's comparison inequality
(Theorem \ref{thm:KahaneInequality}) and the reference family of
kernels $K_{C,\gamma}$ for which uniform integrability is known.
It follows from \eqref{eq:BoundOnK} that there exists a constant
$C_{0}$, such that for every $\varepsilon>0$ there exists $C\left(\varepsilon\right)$,
such that
\[
\forall t,s:K_{\varepsilon}\left(t,s\right)\le\tilde{K}_{C\left(\varepsilon\right),\gamma}\left(t,s\right)+C_{0},
\]
where $\gamma=\sqrt{2d-\delta}$. Now by Theorem \ref{thm:ExistenceClassical}
GMCs with kernels $\tilde{K}_{C\left(\varepsilon\right),\gamma}\left(t,s\right)$,
and thus also $\tilde{K}_{C\left(\varepsilon\right),\gamma}\left(t,s\right)+C_{0}$,
are uniformly integrable. Therefore, by la Vallée Poussin's theorem
and Kahane's inequality (Theorem \ref{thm:KahaneInequality}), GMCs
with kernels $K_{\varepsilon}\left(t,s\right)$ are also uniformly
integrable. Therefore, all the assumptions of Theorem \ref{thm:Approximation}
are verified, and we have convergence $M_{\varepsilon}\overset{L^{0}}{\to}M$.
That $M$ does not depend on the approximation follows from our uniqueness
result (Corollary \ref{cor:Uniqueness}).
\end{proof}

\section{Proof of Theorem \ref{thm:GMCShifts} \label{sec:GMCShifts}}
\begin{proof}
Assume first that there exists a subcritical GMC $M$.

Define a measure $\Q$ on $\OMEGA\times\mathcal{T}$ by \eqref{eq:DefQ}.
We are going to prove \eqref{eq:QShift} by computing the conditional
Laplace transform of $X$ given $t$ under the measure $\Q$.

Let $\xi\in H$ and let $\varphi$ be an arbitrary positive measurable
function on $\mathcal{T}$ (defined $\mu$-almost everywhere). Then

\begin{equation}
\begin{aligned}\E_{\Q}\varphi\left(t\right)\exp\left\langle \xi,X\right\rangle  & \underset{\left(1\right)}{=}\E\intop\varphi\left(t\right)\exp\left\langle \xi,X\right\rangle M\left(X,dt\right)\\
 & \underset{\left(2\right)}{=}\exp\frac{1}{2}\norm{\xi}^{2}\cdot\E\intop\varphi\left(t\right)M\left(X+\xi,dt\right)\\
 & \underset{\left(3\right)}{=}\exp\frac{1}{2}\norm{\xi}^{2}\cdot\E\intop\varphi\left(t\right)\exp\left\langle \xi,Y\left(t\right)\right\rangle M\left(X,dt\right)\\
 & =\E\exp\left\langle \xi,X\right\rangle \cdot\E_{\mu}\varphi\left(t\right)\exp\left\langle \xi,Y\left(t\right)\right\rangle \\
 & =\E_{\P\otimes\mu}\varphi\left(t\right)\exp\left\langle \xi,X+Y\left(t\right)\right\rangle .
\end{aligned}
\label{eq:BunchOfEqualities}
\end{equation}
``$\underset{\left(1\right)}{=}$'' follows from the definition
of $\Q$, ``$\underset{\left(2\right)}{=}$'' is an application
of the Cameron-Martin theorem \cite[Theorem 14.1]{Janson} to the
shift $\xi$, and ``$\underset{\left(3\right)}{=}$'' is the definition
of GMC. The equality of the left-hand side and the right-hand side
of \eqref{eq:BunchOfEqualities} for all $\xi$ and $\varphi$ implies
\eqref{eq:QShift}, since it amounts to equality of conditional Laplace
transforms of $X$ conditioned on $t$, together with the tautology
$\Law_{\Q}t=\Law_{\P\otimes\mu}t=\mu$.

Note that if $M$ exists then \eqref{eq:QShift} implies that $Y$
is a randomized shift. Indeed, $\Law_{\P\otimes\mu}\left[X+Y\right]\ll\Law X$
with density $M\left[\mathcal{T}\right]$.

Conversely, assume that $Y$ is a randomized shift. Then define a
measure $\Q^{\prime}$ on $\OMEGA\times\mathcal{T}$ equipped with
the $\sigma$-algebra $\sigma\left(X,t\right)$ by
\[
\Law_{\Q^{\prime}}\left[X,t\right]:=\Law_{\P\otimes\mu}\left[X+Y\left(t\right),t\right].
\]
The absolute continuity property in the definition of randomized shift
amounts to saying that the $\OMEGA$-projection of $\Q^{\prime}$
is absolutely continuous with respect to $\P$ on $\sigma\left(X\right)$,
so, in particular, one can define a random measure $M^{\prime}\left(X,dt\right)$
via disintegration:
\[
\P\left(d\omega\right)M^{\prime}\left(X,dt\right):=\Q^{\prime}\left(d\omega,dt\right)
\]
(this all happens on $\left(\OMEGA\times\mathcal{T},\sigma\left(X,t\right)\right)$,
so $M^{\prime}$ is automatically measurable with respect to $X$).
$\E M^{\prime}$ is the $\mathcal{T}$-projection of $\Q^{\prime}$,
so $\E M^{\prime}=\mu$ is $\sigma$-finite.

To check that $M^{\prime}$ is a GMC, introduce a measure class preserving
action $S_{\xi},\xi\in H$ of the additive group of $H$ on $\left(\OMEGA\times\mathcal{T},\sigma\left(X,t\right),\Q^{\prime}\right)$
by
\[
S_{\xi}\left(X,t\right):=\left(X+\xi,t\right).
\]
Since it really acts only on $X$ and does not change $t$, there
is also an action $X\mapsto X+\xi$ on $\left(\OMEGA,\sigma\left(X\right),\P\right)$,
which we also denote by $S_{\xi}$.

$S_{\xi}$ is measure class preserving ($\Q^{\prime}$), since it
preserves the measure class of almost all fibers in the disintegration
of $\Q^{\prime}$ with respect to $t$. Indeed, this amounts to saying
that $\Lawcm{\P\otimes\mu}{X+Y\left(t\right)+\xi}t\ll\Lawcm{\P\otimes\mu}{X+Y\left(t\right)}t$
for $\mu$-almost all $t$, which is obvious from the Cameron-Martin
theorem. Moreover, the same argument gives an expression for the density:
\begin{equation}
\frac{\left(S_{\xi}\right)_{\ast}\Q^{\prime}\left(d\omega,dt\right)}{\Q^{\prime}\left(d\omega,dt\right)}=\exp\left[\left\langle X\left(\omega\right)-Y\left(t\right),\xi\right\rangle -\frac{1}{2}\norm{\xi}^{2}\right]\label{eq:Density1}
\end{equation}
($\left(S_{\xi}\right)_{\ast}\Q^{\prime}$ is the pushforward of $\Q^{\prime}$
by the map $S_{\xi}$, i.e. $\left(S_{\xi}\right)_{\ast}\Q^{\prime}\left[A\right]=\Q^{\prime}\left[S_{\xi}^{-1}\left[A\right]\right]$
for measurable sets $A\in\sigma\left(X,t\right)$).

Now we compute the very same density in a different way, by disintegrating
$\Q^{\prime}$ with respect to $X$ instead of $t$. This shows how
$M^{\prime}$ behaves with respect to shifts:
\begin{equation}
\begin{aligned}\frac{\left(S_{\xi}\right)_{\ast}\Q^{\prime}\left(d\omega,dt\right)}{\Q^{\prime}\left(d\omega,dt\right)} & =\frac{\left(S_{\xi}\right)_{\ast}\left[\P\left(d\omega\right)M^{\prime}\left(X\left(\omega\right),dt\right)\right]}{\P\left(d\omega\right)M^{\prime}\left(X\left(\omega\right),dt\right)}\\
 & =\frac{\left(S_{\xi}\right)_{\ast}\P\left(d\omega\right)}{\P\left(d\omega\right)}\cdot\frac{M^{\prime}\left(X\left(\omega\right)-\xi,dt\right)}{M^{\prime}\left(X\left(\omega\right),dt\right)}\\
 & =\exp\left[\left\langle X\left(\omega\right),\xi\right\rangle -\frac{1}{2}\norm{\xi}^{2}\right]\cdot\frac{M^{\prime}\left(X\left(\omega\right)-\xi,dt\right)}{M^{\prime}\left(X\left(\omega\right),dt\right)}.
\end{aligned}
\label{eq:Density2}
\end{equation}
By comparing \eqref{eq:Density1} to \eqref{eq:Density2} we see that
$M^{\prime}$ is a GMC.
\end{proof}

\section{Proof of Theorem \ref{thm:HilbertSchmidtMoments} and Corollary \ref{cor:Moments}
\label{sec:KernelRegularity}}

The following basic observation will be useful in the proof:
\begin{lem}
Let $M_{Y}$ be the subcritical GMC associated to $Y$. Then there
exists a subcritical GMC associated to $cY$ for any $\left|c\right|\le1$,
namely,
\[
M_{cY}\left(X\right):=\Ec{M_{Y}\left(cX+\left(1-c^{2}\right)^{1/2}X^{\prime}\right)}X
\]
where $X^{\prime}$ is an independent copy of $X$. \label{lem:CY}\end{lem}
\begin{proof}
Shifting $X$ by $\xi\in H$ in the left-hand side amounts to shifting
the argument $cX+\left(1-c^{2}\right)^{1/2}X^{\prime}$ in the right-hand
side by $c\xi$, so that the GMC property for $M_{cY}$ follows trivially
from the one for $M_{Y}$.
\end{proof}

\begin{proof}
[Proof of Theorem \ref{thm:HilbertSchmidtMoments}] Fix $n$, and
consider a Hilbert-Schmidt symmetric tensor $\lambda\in H^{\otimes n}$.
Denote the Wick polynomial corresponding to $\lambda$ by 
\[
P_{\lambda}\left(X\right):=\Wick{\left\langle \lambda,X^{\otimes n}\right\rangle }.
\]
It is well-known (see, e.g., \cite{Janson}) that $\E\left|P_{\lambda}\left(X\right)\right|^{2}=n!\norm{\lambda}_{2}^{2}$,
so in particular the family of random variables 
\[
\left\{ P_{\lambda}\left(X\right),\norm{\lambda}_{2}\le1\right\} 
\]
is bounded in probability. Now by Lemma \ref{lem:CY} we know that
for every $\left|c\right|\le1$ the measure $\Law_{\P\otimes\mu}\left[X+cY\right]$
is absolutely continuous with respect to $\Law X$. Therefore for
fixed $c$ the family of random variables 
\[
\left\{ P_{\lambda}\left(X+cY\right)\,\middle|\,\norm{\lambda}_{2}\le1\right\} 
\]
is bounded in probability. Since $c\mapsto P_{\lambda}\left(X+cY\right)$
is an $n$-th degree polynomial, we can extract its $n$-th degree
coefficient in $c$ (denoted by $\left[c^{n}\right]P_{\lambda}\left(X+cY\right)$)
by taking an appropriate linear combination of its values at different
$c$. For instance, we can use the $n$-th iterated difference with
step $\frac{1}{n}$: 
\begin{equation}
\left[c^{n}\right]P_{\lambda}\left(X+cY\right)=\frac{n^{n}}{n!}\sum_{k=0}^{n}{n \choose k}\left(-1\right)^{n-k}P_{\lambda}\left(X+\frac{k}{n}Y\right).\label{eq:nCoeff}
\end{equation}

Next we argue that the random variables $\left\langle \lambda,Y^{\otimes n}\right\rangle $
are well-defined and continuously depend on $\lambda\in H^{\otimes n}$.
To this end note that they are well-defined a priori for $\lambda$
of finite rank (i.e. $\lambda\in H_{\op{alg}}^{\otimes n}$). Also
note that the Wick product and the ordinary product only differ in
lower degree terms, therefore for any finite rank $\lambda$ we have,
\[
\left\langle \lambda,Y^{\otimes n}\right\rangle =\left[c^{n}\right]P_{\lambda}\left(X+cY\right).
\]
Now for fixed $n$, by \eqref{eq:nCoeff}, this family of random variables
is bounded in probability as $\norm{\lambda}_{2}\le1,\rank\lambda<\infty$.
This implies that the $L^{0}$-valued operator $Y^{\otimes n}$, defined
on finite rank tensors, is bounded in the Hilbert-Schmidt norm, so
it can be extended by continuity to an operator defined on all Hilbert-Schmidt
tensors $\lambda$. Therefore, by Theorem \ref{thm:Factorization},
for some measures $\mu_{n}^{\prime}$ equivalent to $\mu$ all linear
functionals of $Y^{\otimes n}$ have second (therefore, first) moments,
and $\norm{\E_{\mu_{n}^{\prime}}Y^{\otimes n}}_{2}<\infty$.
\end{proof}

\begin{proof}
[Proof of Corollary \ref{cor:Moments}] If $Y$ is an $H$-valued
function, we have:
\[
\begin{aligned}\intop\left(K\left(t,s\right)\right)^{n}\mu_{n}^{\prime}\left(dt\right)\mu_{n}^{\prime}\left(ds\right) & =\intop\left\langle Y\left(t\right),Y\left(s\right)\right\rangle ^{n}\mu_{n}^{\prime}\left(dt\right)\mu_{n}^{\prime}\left(ds\right)\\
 & =\intop\left\langle \left(Y\left(t\right)\right)^{\otimes n},\left(Y\left(s\right)\right)^{\otimes n}\right\rangle \mu_{n}^{\prime}\left(dt\right)\mu_{n}^{\prime}\left(ds\right)\\
 & =\norm{\E_{\mu_{n}^{\prime}}Y^{\otimes n}}_{2}^{2}
\end{aligned}
\]
which is finite for the measure $\mu_{n}^{\prime}$ constructed in
Theorem \ref{thm:HilbertSchmidtMoments}. Below we extend this computation
to the case where $Y$ is a generalized random vector.

Fix an orthonormal basis $\left(e_{i}\right)$ in $H$, and denote
by $Y_{i}$ the $i$-th coordinate of $Y$ in this basis, i.e. $Y_{i}=\left\langle Y,e_{i}\right\rangle $.
Fix also a basis $\left(\varepsilon_{p}\right)$ in $L^{2}\left(\mu_{n}^{\prime}\right)$.
Then the Hilbert-Schmidt norm of $K$ may be rewritten as
\[
\begin{aligned}\sum_{p,q}\left\langle \E_{\mu_{n}^{\prime}}\varepsilon_{p}Y,\E_{\mu_{n}^{\prime}}\varepsilon_{q}Y\right\rangle _{H}^{2} & =\sum_{p,q}\left(\sum_{i}\E_{\mu_{n}^{\prime}}\varepsilon_{p}Y_{i}\cdot\E_{\mu_{n}^{\prime}}\varepsilon_{q}Y_{i}\right)^{2}\\
 & =\sum_{p,q}\sum_{i,j}\E_{\mu_{n}^{\prime}}\varepsilon_{p}Y_{i}\cdot\E_{\mu_{n}^{\prime}}\varepsilon_{q}Y_{i}\cdot\E_{\mu_{n}^{\prime}}\varepsilon_{p}Y_{j}\cdot\E_{\mu_{n}^{\prime}}\varepsilon_{q}Y_{j}\\
 & =\sum_{i,j}\left(\sum_{p}\E_{\mu_{n}^{\prime}}\varepsilon_{p}Y_{i}\cdot\E_{\mu_{n}^{\prime}}\varepsilon_{p}Y_{j}\right)^{2}\\
 & =\sum_{i,j}\left(\E_{\mu_{n}^{\prime}}Y_{i}Y_{j}\right)^{2}=\norm{\E_{\mu_{n}^{\prime}}Y^{\otimes2}}_{2}^{2},
\end{aligned}
\]
which is finite by Theorem \ref{thm:HilbertSchmidtMoments}.

It follows by standard theory that
\[
\norm{\intop f\left(t\right)Y\left(t\right)\mu_{n}^{\prime}\left(dt\right)}^{2}=\intop K\left(t,s\right)f\left(t\right)f\left(s\right)\mu_{n}^{\prime}\left(dt\right)\mu_{n}^{\prime}\left(ds\right)
\]
for some $K\in L^{2}\left(\mu_{n}^{\prime}\otimes\mu_{n}^{\prime}\right)$.
Now that $K$ is actually a function, we may express it as the $L^{2}\left(\mu_{n}^{\prime}\otimes\mu_{n}^{\prime}\right)$-convergent
sum
\[
K\left(t,s\right)=\sum_{i}Y_{i}\left(t\right)Y_{i}\left(s\right).
\]
Without limitation of generality we may assume that $n$ is even.
By Fatou's lemma,
\[
\begin{aligned}\intop\left(K\left(t,s\right)\right)^{n}\mu_{n}^{\prime}\left(dt\right)\mu_{n}^{\prime}\left(ds\right) & \le\liminf_{N\to\infty}\intop\left(\sum_{i\le N}Y_{i}\left(t\right)Y_{i}\left(s\right)\right)^{n}\mu_{n}^{\prime}\left(dt\right)\mu_{n}^{\prime}\left(ds\right)\\
 & =\liminf_{N\to\infty}\sum_{i_{1}\le N}\dots\sum_{i_{n}\le N}\left(\E_{\mu_{n}^{\prime}}Y_{i_{1}}\dots Y_{i_{n}}\right)^{2}\\
 & =\norm{\E_{\mu_{n}^{\prime}}Y^{\otimes n}}_{2}^{2},
\end{aligned}
\]
which is finite by Theorem \ref{thm:HilbertSchmidtMoments}.
\end{proof}

\section{Proof of Theorem \ref{thm:Approximation} \label{sec:Approximation}}

\subsection{Outline}

The idea of the proof of Theorem \ref{thm:Approximation} can be outlined
as follows.

By a compactness argument (Lemma \ref{lem:Compactness}) we extract
a subsequence $\left(n^{\prime}\right)$, such the couple $\left(X,M_{Y_{n^{\prime}}}\right)$
converges in law to $\left(X,M\right)$ for some random measure $M$,
defined on an extended probability space and not necessarily measurable
with respect to $X$. By a general argument (Lemma \ref{lem:MeasurabilityImpliesConvergence})
in order for $M_{Y_{n}}$ to converge to $M$ in probability rather
than just in law it is enough to show that $M$ is measurable with
respect to $X$.

On the other hand, we show in Lemma \ref{lem:CoupledGMC} that the
coupling of the limiting measure $M$ to the Gaussian $X$ exhibits
the following property: for every $\xi\in H$ the distribution of
$\left(X+\xi,e^{\left\langle \xi,Y\right\rangle }M\right)$ is absolutely
continuous with respect to that of $\left(X,M\right)$ with the usual
Cameron-Martin density $e^{\left\langle \cdot,\xi\right\rangle -\frac{1}{2}\norm{\xi}^{2}}$
--- in particular, for any bounded continuous function $G:\RR_{+}\to\RR_{+}$
\begin{equation}
\E e^{\left\langle X,\xi\right\rangle -\frac{1}{2}\norm{\xi}^{2}}G\left(M\left[\mathcal{T}\right]\right)=\E G\left(\intop e^{\left\langle \xi,Y\left(s\right)\right\rangle }M\left(ds\right)\right).\label{eq:GtrickXi}
\end{equation}
This property implies, in particular, that $\Ec MX=:M_{Y}$ is a subcritical
GMC associated to $Y$, which is, therefore, a randomized shift.

The heart of the proof consists in establishing the measurability
of $M$ with respect to $X$, or, in other words, that $M=\Ec MX=M_{Y}$.
This is done by replacing the deterministic shift $\xi$ in \eqref{eq:GtrickXi}
by the randomized shift $Y$, in two ways:
\begin{equation}
\begin{aligned}\E M\left[\mathcal{T}\right]G\left(M\left[\mathcal{T}\right]\right) & \underset{(1)}{=}\E\intop\mu\left(dt\right)G\left(\intop e^{\left\langle Y\left(t\right),Y\left(s\right)\right\rangle }M\left(ds\right)\right)\\
 & \underset{(2)}{\le}\E M_{Y}\left[\mathcal{T}\right]G\left(M\left[\mathcal{T}\right]\right).
\end{aligned}
\label{eq:GtrickY}
\end{equation}
This is the content of Claim \ref{cl:Gtrick}. On the one hand, ``$\underset{(1)}{=}$''
is a property of the law of $M$ alone, and it is proved by passing
to the limit the corresponding property for the approximating GMCs
$M_{Y_{n}}$ (Lemma \ref{lem:ExpK}). On the other hand, ``$\underset{(2)}{\le}$''
is a property of the coupling $\left(X,M\right)$ that follows by
randomizing $\xi$ in \eqref{eq:GtrickXi}.

Finally, we apply \eqref{eq:GtrickY} with $G\left(x\right):=\frac{x}{1+x}$.
It follows that $M\left[\mathcal{T}\right]=M_{Y}\left[\mathcal{T}\right]$
by simple manipulations with Jensen's inequality, using the fact that
$G$ is strictly concave and $x\mapsto xG\left(x\right)$ is strictly
convex.

The rest of the section is structured as follows. In Section \ref{sub:Body}
we state explicitly the Lemmas alluded to above and show how Theorem
\ref{thm:Approximation} follows from them. In Section \ref{sub:ExpK}
we prove the central technical Lemma \ref{lem:ExpK}, which allows
to replace deterministic shifts by randomized shifts in the definition
of GMC and is crucial to the proof of \eqref{eq:GtrickY} above. In
Section \ref{sub:USAC} we introduce the notion of uniform stochastic
absolute continuity that is used in the approximation procedures.
In Sections \ref{sub:CoupledGMC}-\ref{sub:Basic} we prove all the
lemmas announced in Section \ref{sub:Body}.

After the proof, in Section \ref{sub:Remark} we briefly discuss another
use of Lemma \ref{lem:ExpK} and the related nonatomicity problem.

\subsection{Key lemmas and proof of Theorem \ref{thm:Approximation} \label{sub:Body}}

Throughout this section we use the notation and assumptions of Theorem
\ref{thm:Approximation}.

We begin with the basic distributional compactness result, Lemma \ref{lem:Compactness}.
The important assumptions for it are the boundedness of expectations
$\E M_{Y_{n}}=\mu$, which ensures that the sequence $\left\{ M_{Y_{n}}\left[\mathcal{T}\right]\right\} $
is tight, and the uniform integrability of $M_{Y_{n}}\left[\mathcal{T}\right]$
which implies that the limiting random measure also has expectation
$\mu$. Although in our setting the space $\mathcal{T}$ carries no
canonical topology, in the proof of Lemma \ref{lem:Compactness} we
reduce it to the classical weak compactness of the space of measures
on a compact metrizable space by introducing an auxiliary topology.
\begin{lem}
There exists a subsequence $\left(n^{\prime}\right)$, and a random
measure $M$ on $\left(\mathcal{T},\mu\right)$ with $\E M=\mu$,
possibly defined on an extended probability space, such that
\begin{equation}
\forall\xi\in H,\forall f\in L^{1}\left(\mu\right):\left(\left\langle X,\xi\right\rangle ,\intop f\left(t\right)M_{Y_{n^{\prime}}}\left(dt\right)\right)\overset{\Law}{\to}\left(\left\langle X,\xi\right\rangle ,\intop f\left(t\right)M\left(dt\right)\right).\label{eq:JointLimit}
\end{equation}
\label{lem:Compactness}
\end{lem}
We abbreviate \eqref{eq:JointLimit} to ``$\left(X,M_{Y_{n^{\prime}}}\right)\overset{\Law}{\to}\left(X,M\right)$''.
In the sequel we restrict attention to a convergent subsequence without
further mention and reserve the letter $M$ for the limiting measure.

Next we reduce the convergence $M_{Y_{n^{\prime}}}\to M$ in probability
to the measurability of $M$ with respect to $X$.
\begin{lem}
Assume that $M$ is measurable with respect to $X$. Then $M_{Y_{n}}\overset{L^{1}}{\to}M$
in the sense that 
\[
\forall f\in L^{1}\left(\mu\right):\intop f\left(t\right)M_{Y_{n}}\left(dt\right)\overset{L^{1}}{\to}\intop f\left(t\right)M\left(dt\right).
\]
\label{lem:MeasurabilityImpliesConvergence}
\end{lem}
The rest of the proof concerns identifying GMC-like properties of
$M$ by passing to the limit the corresponding properties of $M_{Y_{n}}$.
\begin{lem}
For any $\xi\in H$ 
\[
\Law\left[X+\xi,e^{\left\langle Y,\xi\right\rangle }M\right]\ll\Law\left[X,M\right],
\]
and for any positive random variable $g$, measurable with respect
to $\left(X,M\right)$, we have
\begin{equation}
\E g\left(X+\xi,e^{\left\langle Y,\xi\right\rangle }M\right)=\E e^{\left\langle X,\xi\right\rangle -\frac{1}{2}\norm{\xi}^{2}}g\left(X,M\right).\label{eq:CoupledGMC}
\end{equation}
Furthermore, $Y$ is a randomized shift, and $\Ec MX=:M_{Y}$ is the
GMC associated to it. \label{lem:CoupledGMC}
\end{lem}
The property \eqref{eq:CoupledGMC} can be viewed as the natural extension
of the definition of GMC to random measures that are not necessarily
measurable with respect to $X$.

The remaining part of the statement of Theorem \ref{thm:Approximation}
is the measurability of $M$ with respect to $X$, or equivalently,
$M=\Ec MX$. The key role in proving that is played by the following
claim:
\begin{claim}
Let $G:\RR_{+}\to\RR_{+}$ be any bounded continuous function. Then
\[
\E M\left[\mathcal{T}\right]G\left(M\left[\mathcal{T}\right]\right)=\E\intop\mu\left(dt\right)G\left(\intop e^{\left\langle Y\left(t\right),Y\left(s\right)\right\rangle }M\left(ds\right)\right)\le\E M_{Y}\left[\mathcal{T}\right]G\left(M\left[\mathcal{T}\right]\right).
\]
\label{cl:Gtrick}
\end{claim}
Note that a part of the claim is that $\intop e^{\left\langle Y\left(t\right),Y\left(s\right)\right\rangle }M\left(ds\right)<\infty$
almost surely for $\mu$-almost all $t$. This is also contained in
Lemma \ref{lem:ExpK} below.

Finally we show how Theorem \ref{thm:Approximation} follows from
the above lemmas.
\begin{proof}
[Proof of Theorem \ref{thm:Approximation}] We use Lemma \ref{lem:Compactness}
to extract a subsequence $\left(n^{\prime}\right)$ and a limiting
couple $\left(X,M\right)$. Apply Claim \ref{cl:Gtrick} with the
function $G\left(x\right):=\frac{x}{1+x}$ to obtain the following
chain of inequalities:
\begin{equation}
\begin{aligned}\E M\left[\mathcal{T}\right]G\left(\Ec{M\left[\mathcal{T}\right]}X\right) & =\E\left[\Ec{M\left[\mathcal{T}\right]}XG\left(\Ec{M\left[\mathcal{T}\right]}X\right)\right]\\
 & \le\E M\left[\mathcal{T}\right]G\left(M\left[\mathcal{T}\right]\right)\\
 & \underset{(!)}{\le}\E\left[\Ec{M\left[\mathcal{T}\right]}XG\left(M\left[\mathcal{T}\right]\right)\right]\\
 & =\E\left[M\left[\mathcal{T}\right]\Ec{G\left(M\left[\mathcal{T}\right]\right)}X\right]\\
 & \le\E\left[M\left[\mathcal{T}\right]G\left(\Ec{M\left[\mathcal{T}\right]}X\right)\right].
\end{aligned}
\label{eq:TrickJensen}
\end{equation}
The inequality ``$\underset{(!)}{\le}$'' is the content of Claim
\ref{cl:Gtrick}. The first and the last inequalities are both instances
of Jensen's inequality, applied to the strictly convex function $x\mapsto xG\left(x\right)$
and the strictly concave function $G$ respectively. Equality in Jensen's
inequality implies that $M\left[\mathcal{T}\right]=\Ec{M\left[\mathcal{T}\right]}X=M_{Y}\left[\mathcal{T}\right]$,
so that $M\left[\mathcal{T}\right]$ is measurable with respect to
$X$.

The same argument works with $M$ replaced by $f\cdot M$ for any
nonnegative bounded $f$ (and $\mu$ by $f\cdot\mu$ accordingly),
which amounts to replacing $M\left[\mathcal{T}\right]$ by $\intop f\left(t\right)M\left(dt\right)$.
Thus in fact Claim \ref{cl:Gtrick} implies that $M$ is measurable
with respect to $X$, i.e. $M=M_{Y}$.

Now we use Lemma \ref{lem:MeasurabilityImpliesConvergence} to assert
that $M_{Y}$ is a limit in $L^{1}$ of the chosen subsequence. By
Lemma \ref{lem:Compactness}, any subsequence has a convergent subsequence,
and since the limit is the unique (Corollary \ref{cor:Uniqueness})
subcritical GMC $M_{Y}$, it does not depend on the choice of subsequences.
This means that the original sequence converges.
\end{proof}

\subsection{The ``$\exp K$ lemma'' \label{sub:ExpK}}

A central role in the proof of Theorem \ref{thm:Approximation} is
played by Lemma \ref{lem:ExpK} stated below.
\begin{lem}
Let $Z$ and $W$ be randomized shifts, defined on the probability
space $\left(\mathcal{T},\mu\right)$, and let $M_{Z},M_{W}$ be the
corresponding subcritical GMCs with expectation $\mu$. Let $K_{ZW}\left(t,s\right):=\left\langle Z\left(t\right),W\left(s\right)\right\rangle $.
Then:

\begin{equation}
M_{Z}\left(X+W\left(s\right),dt\right)=\exp K_{ZW}\left(t,s\right)\cdot M_{Z}\left(X,dt\right)\ \left(\P\otimes\mu\text{-a.s.}\right)\label{eq:GMCProperty}
\end{equation}
and
\begin{equation}
\E\left[M_{Z}\left(X\right)\otimes M_{W}\left(X\right)\right]=\exp K_{ZW}\cdot\mu\otimes\mu.\label{eq:SecondMoment}
\end{equation}
\label{lem:ExpK}
\end{lem}
Since both covariances $K_{ZZ}$ and $K_{WW}$ are Hilbert-Schmidt
by means of Corollary \ref{cor:Moments}, $K_{ZW}$ is also Hilbert-Schmidt.
Indeed, $\left[\begin{matrix}K_{ZZ} & K_{ZW}\\
K_{WZ} & K_{WW}
\end{matrix}\right]$ is positive definite and the diagonal blocks are Hilbert-Schmidt,
therefore so are the off-diagonal blocks.

We will use Lemma \ref{lem:ExpK} in the case $Z:=W:=Y_{n}$ or $Z:=W:=Y$,
but for the proof of Lemma \ref{lem:ExpK} it is convenient to separate
$Z$ and $W$.

In the case of ``trivial'' GMC (in the sense of Example \ref{ex:TrivialGMC})
both statements of Lemma \ref{lem:ExpK} are elementary, so in general
they may look formally ``obvious''. However, the real content of
Lemma \ref{lem:ExpK} lies in the implied absolute continuity:
\begin{equation}
M_{Z}\left(X+W\left(s\right)\right)\ll M_{Z}\left(X\right)\text{ (\ensuremath{\P\otimes\mu}-a.s.)},\label{eq:GMCPropertyAC}
\end{equation}
\begin{equation}
\E\left[M_{Z}\otimes M_{W}\right]\ll\E M_{Z}\otimes\E M_{W}.\label{eq:SecondMomentAC}
\end{equation}
In particular, \eqref{eq:SecondMomentAC} is a property specific to
GMCs that is far from being true for general random measures.

Another immediate observation following from Lemma \ref{lem:ExpK}
is that $\E\left[M_{Z}\otimes M_{W}\right]$ is $\sigma$-finite ---
since the right-hand side of \eqref{eq:SecondMoment} obviously is.
This measure may, however, fail to be locally finite in any reasonable
sense --- for example, it may assign infinite mass to every product
set $A\times B\subset\mathcal{T}\times\mathcal{T}$, such that $\mu\left[A\right],\mu\left[B\right]>0$.

In the proof of Lemma \ref{lem:ExpK} we will need the following general
fact:
\begin{lem}
Let $Z_{n}$ and $Z$ be randomized shifts of $X$ defined on the
same probability space $\left(\mathcal{T},\mu\right)$. Assume that
the family of random variables $\left\{ \intop_{\mathcal{T}}M_{Z_{n}}\left(X,dt\right),n\ge1\right\} $
is uniformly integrable. Also assume that $Z_{n}\to Z$ in the following
sense:
\[
\forall\xi\in H:\left\langle Z_{n},\xi\right\rangle \overset{L^{0}\left(\mathcal{T},\mu\right)}{\to}\left\langle Z,\xi\right\rangle .
\]
Then for every $f\left(X\right)\in L^{0}\left(\OMEGA,\sigma\left(X\right),\P\right)$
we have
\[
f\left(X+Z_{n}\right)\overset{L^{0}\left(\left(\OMEGA,\P\right)\otimes\left(\mathcal{T},\mu\right)\right)}{\to}f\left(X+Z\right).
\]
\label{lem:ContinuityRandomized}\end{lem}
\begin{proof}
For $f\left(X\right):=\exp i\left\langle \xi,X\right\rangle $, $\xi\in H$,
this follows from the assumptions. It is well-known that linear combinations
of exponentials $\exp i\left\langle \xi,X\right\rangle $ are dense
in $L^{0}\left(\OMEGA,\sigma\left(X\right)\right)$. Let $f\in L^{0}\left(\OMEGA\right)$
and let $f_{n}\left(X\right)\to f\left(X\right)$ be a sequence of
linear combinations of exponentials approximating $f\left(X\right)$
in $L^{0}\left(\OMEGA\right)$.

We claim that
\begin{equation}
f_{n}\left(X+Z_{m}\right)\overset{L^{0}}{\to}f\left(X+Z_{m}\right),n\to\infty,\text{ uniformly in }m.\label{eq:UniformL0}
\end{equation}
Indeed, for any $\varepsilon$ we have
\[
\P\left\{ \left|f_{n}\left(X\right)-f\left(X\right)\right|>\varepsilon\right\} \to0,n\to\infty.
\]
Due to the uniform integrability of $\left\{ M_{Z_{m}}\left[\mathcal{T}\right]\right\} $,
this implies
\[
\sup_{m}\E M_{Z_{m}}\left[\mathcal{T}\right]\I\left\{ \left|f_{n}\left(X\right)-f\left(X\right)\right|>\varepsilon\right\} \to0,n\to\infty,
\]
so that
\[
\sup_{m}\E_{\P\otimes\nu}\I\left\{ \left|f_{n}\left(X+Z_{m}\right)-f\left(X+Z_{m}\right)\right|>\varepsilon\right\} \to0,n\to\infty,
\]
which implies \eqref{eq:UniformL0}. Similarly, we have 
\begin{equation}
f_{n}\left(X+Z\right)\overset{L^{0}}{\to}f\left(X+Z\right).\label{eq:fnf}
\end{equation}

Now by a standard argument \eqref{eq:UniformL0}, \eqref{eq:fnf}
and the fact that 
\[
\forall n:f_{n}\left(X+Z_{m}\right)\overset{L^{0}}{\to}f_{n}\left(X+Z\right),m\to\infty
\]
implies the statement of the lemma.
\end{proof}

\begin{proof}
[Proof of Lemma \ref{lem:ExpK}] As a first step, we prove \eqref{eq:GMCProperty}
under the assumption that
\[
\left(1+\delta\right)W\text{ is a randomized shift for some }\delta>0,
\]
which we refer to as ``strict subcriticality''.

We approximate the shift $W$ by its projections $P_{n}W$, where
$\left(P_{n}\right)$ is an increasing sequence of finite-dimensional
orthogonal projection operators in $H$, converging strongly to $1$.
Note that for $P_{n}W$ both statements of Lemma \ref{lem:ExpK} are
satisfied by the definition of GMC $M_{Z}$, since $P_{n}$ are finite-dimensional
and $P_{n}W$ are just $H$-valued functions.

Let $f\in L^{\infty}\left(\mathcal{T},\mu\right)$ be a positive function.
By applying Lemma \ref{lem:ContinuityRandomized} to $W_{n}:=cP_{n}W$,
$c\in\left\{ 1,1+\delta\right\} $ , we get:
\[
\intop f\left(t\right)M_{Z}\left(X+cP_{n}W\left(s\right),dt\right)\overset{L^{0}\left(\P\otimes\mu\right)}{\to}\intop f\left(t\right)M_{Z}\left(X+cW\left(s\right),dt\right).
\]
In the left-hand side, since $P_{n}$ is finite-dimensional, 
\begin{equation}
M_{Z}\left(X+cP_{n}W\left(s\right),dt\right)=e^{c\left\langle P_{n}W\left(s\right),Z\left(t\right)\right\rangle }M_{Z}\left(X,dt\right).\label{eq:ExpKFiniteDim}
\end{equation}
By applying \eqref{eq:ExpKFiniteDim} to $c=1+\delta$, we see that,
along a deterministic subsequence, the integrals
\[
\intop e^{\left(1+\delta\right)\left\langle P_{n}W\left(s\right),Z\left(t\right)\right\rangle }M_{Z}\left(X,dt\right)
\]
converge $\P\otimes\mu$-almost surely, and thus stay $\P\otimes\mu$-almost
surely bounded as $n$ increases. Therefore, the subsequence of functions
$t\mapsto e^{\left\langle P_{n}W\left(s\right),Z\left(t\right)\right\rangle }$
is $\P\otimes\mu$-almost surely uniformly integrable against the
random measure $M_{Z}$. Therefore, by the Lebesgue convergence theorem,
along that subsequence
\[
\intop f\left(t\right)e^{\left\langle P_{n}W\left(s\right),Z\left(t\right)\right\rangle }M_{Z}\left(X,dt\right)\to\intop f\left(t\right)e^{\left\langle W\left(s\right),Z\left(t\right)\right\rangle }M_{Z}\left(X,dt\right).
\]
This implies that $M_{Z}\left(X+W\left(s\right),dt\right)=e^{\left\langle W\left(s\right),Z\left(t\right)\right\rangle }M_{Z}\left(X,dt\right)$.

Next we reduce the general case to the strictly subcritical one. For
this we approximate a general randomized shift $W$ by shifts $\left(1-\varepsilon\right)W$
as $\varepsilon\to0$. All $\left(1-\varepsilon\right)W$ are obviously
strictly subcritical, so we already know that
\[
M_{Z}\left(X+\left(1-\varepsilon\right)W\right)=e^{\left(1-\varepsilon\right)K_{ZW}}M_{Z}\left(X\right).
\]
Now on the one hand, by Lemma \ref{lem:ContinuityRandomized}, for
any bounded $f\ge0$
\[
\intop f\left(t\right)M_{Z}\left(X+\left(1-\varepsilon\right)W,dt\right)\overset{L^{0}\left(\P\otimes\mu\right)}{\to}\intop f\left(t\right)M_{Z}\left(X+W,dt\right),\varepsilon\to0.
\]
On the other hand,
\[
\intop f\left(t\right)e^{\left(1-\varepsilon\right)K_{ZW}}M_{Z}\left(X,dt\right)\to\intop f\left(t\right)e^{K_{ZW}}M_{Z}\left(X,dt\right)
\]
almost surely. Indeed, one can split $\intop f\left(t\right)e^{\left(1-\varepsilon\right)K_{ZW}}M_{Z}\left(X,dt\right)$
into the integral over $\left\{ K_{ZW}\ge0\right\} $ and the integral
over $\left\{ K_{ZW}<0\right\} $; on $\left\{ K_{ZW}\ge0\right\} $
the integrand increases as $\varepsilon\downarrow0$, so the monotone
convergence theorem applies, and on $\left\{ K_{ZW}<0\right\} $ it
is dominated by $1$, which is integrable against $M_{Z}$. Thus we
have proved \eqref{eq:GMCProperty} in the general case.

\eqref{eq:SecondMoment} is deduced from \eqref{eq:GMCProperty} by
taking the expectation of both sides multiplied by $\mu\left(ds\right)$.
Indeed, on the one hand, by Theorem \ref{thm:GMCShifts},
\[
\E\left[M_{Z}\left(X+W\left(s\right),dt\right)\mu\left(ds\right)\right]=\E\left[M_{Z}\left(X,dt\right)\otimes M_{W}\left(X,ds\right)\right],
\]
and on the other hand,
\[
\E\left[\exp K_{ZW}\left(t,s\right)\cdot M_{Z}\left(X,dt\right)\mu\left(ds\right)\right]=\exp K_{ZW}\left(t,s\right)\cdot\mu\left(dt\right)\otimes\mu\left(ds\right).
\]

\end{proof}

\subsection{Uniform stochastic absolute continuity \label{sub:USAC}}

In order to deal with the approximations of random measures involved
in the proof of Claim \ref{cl:Gtrick} we need a measure-theoretic
tool (Lemma \ref{lem:StochasticLebesgue}) for proving convergence
of integrals of functions against random measures. This lemma can
be seen as a stochastic analogue of Lebesgue's convergence theorem,
and just like Lebesgue's theorem, it comes with a related notion of
``uniform integrability''.
\begin{defn}
A family of random measures $\left\{ M_{\alpha}\right\} $ on a measurable
space $\mathcal{T}$ is called \emph{uniformly stochastically absolutely
continuous} with respect to a deterministic probability measure $\mu$
on $\mathcal{T}$ if
\begin{itemize}
\item $\forall\alpha:\E M_{\alpha}\ll\mu$
\item For every $c>0$ we have
\[
\sup_{\substack{A\subset\mathcal{T}\\
\mu\left[A\right]\le\varepsilon
}
}\sup_{\alpha}\P\left\{ M_{\alpha}\left[A\right]>c\right\} \to0,\varepsilon\to0.
\]

\end{itemize}
\end{defn}
\begin{example}
If $\E M_{\alpha}=\mu$ for all $\alpha$ then $\left\{ M_{\alpha}\right\} $
is uniformly stochastically absolutely continuous. Indeed, if $\mu\left[A\right]<\varepsilon$
then $M_{\alpha}\left[A\right]$ are uniformly small in $L^{1}$,
therefore uniformly small in probability. \label{ex:USACL1}\end{example}
\begin{lem}
Let $\left(M_{n}\right)$ be a sequence of random measures on $\mathcal{T}$,
such that $\forall\alpha:\E M_{\alpha}\ll\mu$, and let $F_{n},F\in L^{0}\left(\mathcal{T},\mu\right)$.
Assume that:
\begin{itemize}
\item $M_{n}\to M$ in the sense that
\begin{equation}
\forall f\in L^{\infty}\left(\mathcal{T},\mu\right):\intop f\left(t\right)M_{n}\left(dt\right)\overset{L^{0}}{\to}\intop f\left(t\right)M\left(dt\right);\label{eq:ConvergenceOfRandomMeasures}
\end{equation}

\item $F_{n}\overset{L^{0}\left(\mu\right)}{\to}F$
\item For all $n$ we have
\[
\intop\left|F_{n}\left(t\right)\right|M_{n}\left(dt\right)<\infty;
\]

\item The family of random measures $\left\{ \left|F_{n}\left(t\right)\right|M_{n}\left(dt\right)\right\} $
is uniformly stochastically absolutely continuous with respect to
$\mu$.
\end{itemize}
Then
\[
\intop\left|F\left(t\right)\right|M\left(dt\right)<\infty
\]
and
\begin{equation}
\intop F_{n}\left(t\right)M_{n}\left(dt\right)\overset{L^{0}}{\to}\intop F\left(t\right)M\left(dt\right).\label{eq:ConvergenceOfIntegrals}
\end{equation}

The same is true if we replace convergence in $L^{0}$ by convergence
in law both in \eqref{eq:ConvergenceOfRandomMeasures} and in \eqref{eq:ConvergenceOfIntegrals}.
\label{lem:StochasticLebesgue}\end{lem}
\begin{proof}
By passing to a subsequence, we may assume that $F_{n}\to F$ almost
everywhere. Fix $\varepsilon>0$. By Egorov's theorem, there is a
set $A_{\varepsilon}\subset\mathcal{T}$, such that $\mu\left[\mathcal{T}\setminus A_{\varepsilon}\right]\le\varepsilon$
and $F_{n}\to F$ uniformly on $A_{\varepsilon}$.
\begin{multline*}
\left|\intop F_{n}\left(t\right)M_{n}\left(dt\right)-\intop_{A_{\varepsilon}}F\left(t\right)M_{n}\left(dt\right)\right|\\
\le\intop_{\mathcal{T}\setminus A_{\varepsilon}}\left|F_{n}\left(t\right)\right|M_{n}\left(dt\right)+\mathop{\mathrm{ess}\sup}_{A_{\varepsilon}}\left|F_{n}-F\right|\cdot M_{n}\left[\mathcal{T}\right].
\end{multline*}
The first term is small in probability uniformly in $n$ whenever
$\varepsilon$ is small due to the uniform stochastic absolute continuity.
The second term is small in probability for fixed $\varepsilon$ and
large $n$ because the family of random variables $\left\{ M_{n}\left[\mathcal{T}\right]\right\} $
is bounded in probability and $\mathop{\mathrm{ess}\sup}_{A_{\varepsilon}}\left|F_{n}-F\right|\to0,n\to\infty$.

On the other hand, $\intop_{A_{\varepsilon}}F\left(t\right)M_{n}\left(dt\right)$
is close in probability (resp. in law) to $\intop_{A_{\varepsilon}}F\left(t\right)M\left(dt\right)$
by assumption.
\end{proof}
The next statement is based entirely on Lemma \ref{lem:ExpK}.
\begin{lem}
Let $\left\{ Y_{\alpha}\right\} $ and $\left\{ Z_{\beta}\right\} $
be families of randomized shifts on $\left(\mathcal{T},\mu\right)$,
and let $\left\{ M_{Y_{\alpha}}\left[\mathcal{T}\right]\right\} $
be uniformly integrable. Let $T$ be a random point in $\mathcal{T}$
with law $\mu$, independent of $X$. Then the family of random measures
$N_{\alpha\beta}$ on $\mathcal{T}\times\mathcal{T}$ given by
\[
N_{\alpha\beta}:=\exp K_{Y_{\alpha}Z_{\beta}}\left(T,\cdot\right)\cdot\delta_{T}\otimes M_{Z_{\beta}}\left(X\right)
\]
is uniformly stochastically absolutely continuous with respect to
$\mu\otimes\mu$. \label{lem:USACGMC}\end{lem}
\begin{proof}
Fix $\varepsilon>0$. Consider a measurable subset $A\subset\mathcal{T}\times\mathcal{T}$,
and denote by $A_{t},t\in\mathcal{T}$ its $t$-section, i.e. 
\[
A_{t}:=\left\{ s\in\mathcal{T}:\left(t,s\right)\in A\right\} \subset\mathcal{T}.
\]
By Lemma \ref{lem:ExpK}
\[
\begin{aligned}\P\left\{ N_{\alpha\beta}\left[A\right]>c\right\}  & =\P\left\{ \intop_{A_{T}}\exp K_{Y_{\alpha}Z_{\beta}}\left(T,s\right)M_{Z_{\beta}}\left(X,ds\right)>c\right\} \\
 & =\P\left\{ \intop_{A_{T}}M_{Z_{\beta}}\left(X+Y_{\alpha}\left(T\right),ds\right)>c\right\} \\
 & =\E\intop M_{Y_{\alpha}}\left(X,dt\right)\I\left\{ \intop_{A_{t}}M_{Z_{\beta}}\left(X,ds\right)>c\right\} \\
 & \le\mu\left\{ t:\nu\left[A_{t}\right]>\varepsilon\right\} \\
 & \quad+\E\intop M_{Y_{\alpha}}\left(X,dt\right)\I\left\{ \nu\left[A_{t}\right]\le\varepsilon\right\} \I\left\{ \intop_{A_{t}}M_{Z_{\beta}}\left(X,ds\right)>c\right\} .
\end{aligned}
\]
The first term is small whenever $\left(\mu\otimes\nu\right)\left[A\right]$
is small enough. The second term is small whenever $\varepsilon$
is small enough, since for $t$, such that $\nu\left[A_{t}\right]\le\varepsilon$,
we have 
\[
\P\left\{ \intop_{A_{t}}M_{Z_{\beta}}\left(X,ds\right)>c\right\} \le c^{-1}\E\intop_{A_{t}}M_{Z_{\beta}}\left(X,ds\right)=c^{-1}\nu\left[A_{t}\right]\le c^{-1}\varepsilon,
\]
and $\left\{ \intop M_{Y_{\alpha}}\left(X,dt\right)\right\} $ is
assumed to be uniformly integrable.
\end{proof}

\subsection{Proof of Lemma \ref{lem:CoupledGMC} \label{sub:CoupledGMC}}
\begin{proof}
It is enough to check \eqref{eq:CoupledGMC} for functions $g\left(X,M\right)$
of the form
\[
g\left(X,M\right):=\exp\left[\left\langle X,\eta\right\rangle -\frac{1}{2}\norm{\eta}^{2}\right]\cdot h\left(M\right)
\]
for all $\eta\in H$ and all bounded measurable $h$ that depend continuously
on finitely many integrals $\intop f_{1}\left(t\right)M\left(dt\right),\dots,\intop f_{m}\left(t\right)M\left(dt\right)$.
First we consider the case $\eta=0$. For the functions of such form,
it follows from the convergence in law $\left(X,M_{Y_{n}}\right)\overset{\Law}{\to}\left(X,M\right)$
that
\[
\E e^{\left\langle X,\xi\right\rangle -\frac{1}{2}\norm{\xi}^{2}}h\left(M_{Y_{n}}\right)\to\E e^{\left\langle X,\xi\right\rangle -\frac{1}{2}\norm{\xi}^{2}}h\left(M\right).
\]
On the other hand, by the Cameron-Martin formula and the definition
of GMC,
\[
\E e^{\left\langle X,\xi\right\rangle -\frac{1}{2}\norm{\xi}^{2}}h\left(M_{Y_{n}}\right)=\E h\left(M_{Y_{n}}\left(X+\xi\right)\right)=\E h\left(e^{\left\langle \xi,Y_{n}\right\rangle }M_{Y_{n}}\right).
\]
The random measures $e^{\left\langle \xi,Y_{n}\right\rangle }M_{Y_{n}}$
are uniformly stochastically absolutely continuous by a degenerate
special case of Lemma \ref{lem:USACGMC} where one of the shifts is
the deterministic shift $\xi$. Thus by Lemma \ref{lem:StochasticLebesgue}
applied to the measures $M_{Y_{n}}\overset{\Law}{\to}M$ we have
\[
\E h\left(e^{\left\langle \xi,Y_{n}\right\rangle }M_{Y_{n}}\right)\to\E h\left(e^{\left\langle \xi,Y\right\rangle }M\right).
\]
Therefore,
\begin{equation}
\E e^{\left\langle X,\xi\right\rangle -\frac{1}{2}\norm{\xi}^{2}}h\left(M\right)=\E h\left(e^{\left\langle Y,\xi\right\rangle }M\right).\label{eq:CoupledGMCEta0}
\end{equation}

Now consider the case where $\eta$ is not necessarily $0$. \eqref{eq:CoupledGMC}
amounts to proving that
\[
\E e^{\left\langle X,\xi+\eta\right\rangle -\frac{1}{2}\norm{\xi}^{2}-\frac{1}{2}\norm{\eta}^{2}}h\left(M\right)=\E e^{\left\langle X+\xi,\eta\right\rangle -\frac{1}{2}\norm{\eta}^{2}}h\left(e^{\left\langle Y,\xi\right\rangle }M\right).
\]
By \eqref{eq:CoupledGMCEta0}, the left-hand side above equals $e^{\left\langle \xi,\eta\right\rangle }\E h\left(e^{\left\langle Y,\xi+\eta\right\rangle }M\right)$.
The right-hand side is the same by \eqref{eq:CoupledGMCEta0} applied
to the function $\tilde{h}\left(M\right):=h\left(e^{\left\langle Y,\xi\right\rangle }M\right)$.
\end{proof}

\subsection{Proof of Claim \ref{cl:Gtrick}}

Claim \ref{cl:Gtrick} consists of two statements which we prove separately:

\begin{equation}
\E M\left[\mathcal{T}\right]G\left(M\left[\mathcal{T}\right]\right)=\E_{\P\otimes\mu}G\left(\intop\exp\left\langle Y\left(t\right),Y\left(s\right)\right\rangle M\left(ds\right)\right),\label{eq:TrickEq1}
\end{equation}
\begin{equation}
\E M_{Y}\left[\mathcal{T}\right]G\left(M\left[\mathcal{T}\right]\right)\ge\E_{\P\otimes\mu}G\left(\intop\exp\left\langle Y\left(t\right),Y\left(s\right)\right\rangle M\left(ds\right)\right).\label{eq:TrickEq2}
\end{equation}

\begin{proof}
[Proof of Claim \ref{cl:Gtrick}: \eqref{eq:TrickEq1}]

By uniform integrability of $\left\{ M_{Y_{n}}\left[\mathcal{T}\right]\right\} $
and convergence $M_{Y_{n}}\left[\mathcal{T}\right]\overset{\Law}{\to}M\left[\mathcal{T}\right]$
we have
\[
\E M\left[\mathcal{T}\right]G\left(M\left[\mathcal{T}\right]\right)=\lim_{n\to\infty}\E M_{Y_{n}}\left[\mathcal{T}\right]G\left(M_{Y_{n}}\left[\mathcal{T}\right]\right).
\]

By Theorem \ref{thm:GMCShifts} and Lemma \ref{lem:ExpK}, both applied
to a single randomized shift $Y_{n}$, we have
\[
\begin{aligned}\E M_{Y_{n}}\left[\mathcal{T}\right]G\left(M_{Y_{n}}\left[\mathcal{T}\right]\right) & =\E_{\P\otimes\mu}G\left(M_{Y_{n}}\left(X+Y_{n}\left(t\right)\right)\left[\mathcal{T}\right]\right)\\
 & =\E_{\P\otimes\mu}G\left(\intop\exp K_{Y_{n}Y_{n}}\left(t,s\right)M_{Y_{n}}\left(X,ds\right)\right).
\end{aligned}
\]
Thus to prove the claim it is enough to show that for a random point
$T$ in $\mathcal{T}$ with distribution $\mu$, independent of $X$,
we have 
\begin{equation}
\intop\exp K_{Y_{n}Y_{n}}\left(T,s\right)M_{Y_{n}}\left(X,ds\right)\overset{\Law}{\to}\intop\exp K_{YY}\left(T,s\right)M\left(ds\right).\label{eq:ConvergenceExpKM}
\end{equation}
We rewrite both integrals tautologically in a way that involves random
measures and deterministic (i.e. not dependent on $T$) functions:
\[
\intop_{\mathcal{T}}\exp K_{Y_{n}Y_{n}}\left(T,s\right)M_{Y_{n}}\left(X,ds\right)=\intop_{\mathcal{T}\times\mathcal{T}}\exp K_{Y_{n}Y_{n}}\left(t,s\right)\cdot\delta_{T}\left(dt\right)\otimes M_{Y_{n}}\left(X,ds\right),
\]
\[
\intop_{\mathcal{T}}\exp K_{YY}\left(T,s\right)M\left(ds\right)=\intop_{\mathcal{T}\times\mathcal{T}}\exp K_{YY}\left(t,s\right)\cdot\delta_{T}\left(dt\right)\otimes M\left(ds\right).
\]
Now we apply Lemma \ref{lem:USACGMC} to the randomized shifts $\left\{ Y_{\alpha}\right\} :=\left\{ Y_{n}\right\} ,\left\{ Z_{\beta}\right\} :=\left\{ Y_{n}\right\} $
and deduce that $\left\{ \exp K_{Y_{n}Y_{n}}\cdot\delta_{T}\otimes M_{Y_{n}}\right\} $
is uniformly stochastically absolutely continuous. Then apply Lemma
\ref{lem:StochasticLebesgue} to the random measures $\delta_{T}\otimes M_{Y_{n}}$
and functions $\exp K_{Y_{n}Y_{n}}$. Note that $M_{Y_{n}}\overset{\Law}{\to}M$
trivially implies $\delta_{T}\otimes M_{Y_{n}}\overset{\Law}{\to}\delta_{T}\otimes M$,
and that by the assumptions of Theorem \ref{thm:Approximation}, $\exp K_{Y_{n}Y_{n}}\overset{L^{0}\left(\mu\otimes\mu\right)}{\to}\exp K_{YY}$,
so indeed Lemma \ref{lem:StochasticLebesgue} is applicable in this
case, yielding \eqref{eq:ConvergenceExpKM}, and therefore also the
claim.
\end{proof}

\begin{proof}
[Proof of Claim \ref{cl:Gtrick}: \eqref{eq:TrickEq2}] The strategy
is to randomize the $\xi$ in Lemma \ref{lem:CoupledGMC} and thus
approximate the randomized shift $Y$. By applying \eqref{eq:CoupledGMC}
conditionally, we have for every measurable vector-valued function
$\xi:\mathcal{T}\to H$
\[
\E_{\P\otimes\mu}\exp\left[\left\langle \xi\left(t\right),X\right\rangle -\frac{1}{2}\norm{\xi\left(t\right)}^{2}\right]G\left(M\left[\mathcal{T}\right]\right)=\E_{\P\otimes\mu}G\left(\intop\exp\left\langle \xi\left(t\right),Y\left(s\right)\right\rangle M\left(ds\right)\right).
\]

Take any increasing sequence $\left(P_{n}\right)$ of finite-dimensional
projections in $H$ that converge strongly to $1$. Since $P_{n}$
has finite-dimensional range, $P_{n}Y$ is in fact a vector-valued
function. Therefore, we can take $\xi\left(t\right):=P_{n}Y\left(t\right)$
above and obtain
\begin{multline*}
\E_{\P\otimes\mu}\exp\left[\left\langle P_{n}Y\left(t\right),X\right\rangle -\frac{1}{2}\norm{P_{n}Y\left(t\right)}^{2}\right]G\left(M\left[\mathcal{T}\right]\right)\\
=\E_{\P\otimes\mu}G\left(\intop\exp\left\langle P_{n}Y\left(t\right),Y\left(s\right)\right\rangle M\left(ds\right)\right).
\end{multline*}
Since the kernel $\left\langle Y\left(t\right),Y\left(s\right)\right\rangle $
is Hilbert-Schmidt, the function $\left\langle P_{n}Y\left(t\right),Y\left(s\right)\right\rangle $
converges in measure ($\mu\otimes\mu$) to $\left\langle Y\left(t\right),Y\left(s\right)\right\rangle $.
Therefore, by Fatou's lemma
\begin{multline}
\E_{\P\otimes\mu}G\left(\intop\exp\left\langle Y\left(t\right),Y\left(s\right)\right\rangle M\left(ds\right)\right)\\
\le\liminf_{n\to\infty}\E_{\P\otimes\mu}G\left(\intop\exp\left\langle P_{n}Y\left(t\right),Y\left(s\right)\right\rangle M\left(ds\right)\right).\label{eq:EqTrick2LimInf}
\end{multline}
On the other hand, $t$ and $M\left[\mathcal{T}\right]$ are conditionally
independent given $X$, so
\begin{multline*}
\E_{\P\otimes\mu}\exp\left[\left\langle P_{n}Y\left(t\right),X\right\rangle -\frac{1}{2}\norm{P_{n}Y\left(t\right)}^{2}\right]G\left(M\left[\mathcal{T}\right]\right)\\
=\E\left[\intop\exp\left[\left\langle P_{n}Y\left(t\right),X\right\rangle -\frac{1}{2}\norm{P_{n}Y\left(t\right)}^{2}\right]\mu\left(dt\right)\cdot G\left(M\left[\mathcal{T}\right]\right)\right]\\
=\E M_{P_{n}Y}\left[\mathcal{T}\right]G\left(M\left[\mathcal{T}\right]\right).
\end{multline*}

Since $M_{P_{n}Y}\left[\mathcal{T}\right]=\Ec{M_{Y}\left[\mathcal{T}\right]}{P_{n}X}$
is a uniformly integrable martingale that converges to $M_{Y}\left[\mathcal{T}\right]$,
we have
\[
\liminf_{n\to\infty}\E_{\P\otimes\mu}\exp\left[\left\langle P_{n}Y\left(t\right),X\right\rangle -\frac{1}{2}\norm{P_{n}Y\left(t\right)}^{2}\right]G\left(M\left[\mathcal{T}\right]\right)=\E M_{Y}\left[\mathcal{T}\right]G\left(M\left[\mathcal{T}\right]\right).
\]
Together with \eqref{eq:EqTrick2LimInf}, this proves the claim.
\end{proof}

\subsection{Proof of Lemmas \ref{lem:Compactness} and \ref{lem:MeasurabilityImpliesConvergence}
\label{sub:Basic}}
\begin{proof}
[Proof of Lemma \ref{lem:Compactness}]

By our assumptions $\left(\mathcal{T},\mu\right)$ is a standard measure
space, so we may assume that $\mathcal{T}=\left[0,1\right]$ with
its Borel $\sigma$-algebra. We can also identify $X$ (and other
generalized random vectors in $H$) with a random element in the Polish
space $\RR^{\infty}$, equipped with the isomorphism $H\simeq\ell^{2}\subset\RR^{\infty}$.

The family $\left\{ M_{Y_{n}}\left[\mathcal{T}\right]\right\} $ is
tight, therefore the family $\left\{ \left(X,M_{\alpha}\right)\right\} $
of random elements of $\RR^{\infty}\times\op{Measures}\left(\mathcal{T}\right)$
is tight when the space of measures is equipped with the weak topology.
Thus for some subsequence $\left(n^{\prime}\right)$ there is a distributional
limit $\left(X,M\right)$ of $\left(X,M_{Y_{n^{\prime}}}\right)$,
possibly on an extended probability space. This implies \eqref{eq:JointLimit}
for continuous $f$ and $\xi\in H$ with finitely many nonzero coordinates.

It is easy to see that the family of maps 
\[
H\times L^{1}\left(\mu\right)\to L^{2}\left(\P\right)\times L^{1}\left(\P\right),
\]
\[
\left(\xi,f\right)\mapsto\left(\left\langle X,\xi\right\rangle ,\intop f\left(t\right)M_{Y_{n^{\prime}}}\left(dt\right)\right),
\]
 is equicontinuous, so \eqref{eq:JointLimit} follows for all $\xi\in H,f\in L^{1}$.

The equality $\E M=\mu$ follows from $\E M_{Y_{n^{\prime}}}=\mu$
together with the uniform integrability of $\left\{ M_{Y_{n}}\left[\mathcal{T}\right]\right\} $.
\end{proof}

\begin{proof}
[Proof of Lemma \ref{lem:MeasurabilityImpliesConvergence}]

Consider the sequence of triples $\left(X,M_{Y_{n}},M\right)$. By
the same reasoning as in Lemma \ref{lem:Compactness}, it has a subsequential
distributional limit. Since $\left(X,M_{Y_{n}}\right)\overset{\Law}{\to}\left(X,M\right)$,
this distributional limit has the form $\left(X,M,M^{\prime}\right)$,
where $\left(X,M\right)$ and $\left(X,M^{\prime}\right)$ have the
same joint distribution. But since $M$ (and therefore also $M^{\prime}$)
is a function of $X$, we have $M=M^{\prime}$. This implies, in particular,
that $\left(M_{Y_{n}},M\right)\overset{\Law}{\to}\left(M,M\right)$,
so for every function $f\in L^{1}\left(\mu\right)$ 
\[
\intop f\left(t\right)M_{Y_{n}}\left(dt\right)-\intop f\left(t\right)M\left(dt\right)\overset{\Law}{\to}\intop f\left(t\right)M\left(dt\right)-\intop f\left(t\right)M\left(dt\right)=0.
\]
Convergence in law to a constant is equivalent to convergence in probability,
so $\intop f\left(t\right)M_{Y_{n}}\left(dt\right)\overset{L^{0}}{\to}\intop f\left(t\right)M\left(dt\right)$.
By uniform integrability, this upgrades to convergence in $L^{1}$.
\end{proof}

\subsection{A remark on the ``$\exp K$ lemma'' \label{sub:Remark}}

Lemma \ref{lem:ExpK}, which is of central importance in the proof
of Theorem \ref{thm:Approximation}, has other uses as well and may
be of independent interest. In particular, note that the property
\eqref{eq:SecondMomentAC} with $Z=W$ implies immediately that if
$\mu$ has no atoms then almost surely $M_{Z}$ has no atoms. Indeed:
\[
\E\sum_{t\in\left\{ \text{atoms of }M_{Z}\right\} }\left(M_{Z}\left\{ t\right\} \right)^{2}=\E\intop_{\op{diag}\mathcal{T}}M_{Z}\otimes M_{Z}=\intop_{\diag\mathcal{T}}\E\left[M_{Z}\otimes M_{Z}\right]=0,
\]
where $\diag\mathcal{T}:=\left\{ \left(t,t\right)\midmid t\in\mathcal{T}\right\} \subset\mathcal{T}\times\mathcal{T}$.
The measure $\E\left[M_{Z}\otimes M_{Z}\right]$ may ``explode along
the diagonal'', yet it assigns mass $0$ to it.

Recently nonatomicity was proven for some \emph{critical} GMCs ---
namely, those over critical ($\gamma^{2}=2d$) logarithmic fields
\cite{DSRVCriticalConvergence} and the related hierarchical fields,
also known as multiplicative cascades \cite{BKNSWCritical}. In the
general case (i.e. without the subcriticality assumption) nonatomicity
remains an open question.

\section{Acknowledgments}

The author wishes to thank Ofer Zeitouni for many valuable comments
that helped shape the manuscript.

The work is partially supported by Israel Science Foundation grants
111/11 and 147/15.

\appendix

\section{Appendix: the Maurey-Nikishin factorization theorem}

The following factorization theorem follows trivially from the $q:=2$
case of \cite[Théorème 3 b)]{Maurey2}:
\begin{thm}
[Nikishin, Maurey] Let $H$ be a Hilbert space, let $\left(\mathcal{T},\mu\right)$
be a standard probability space. Let $Y:H\to L^{0}\left(\mu\right)$
be a continuous linear operator. Then there exists a probability measure
$\mu^{\prime}$ equivalent to $\mu$, and a bounded operator $Y^{\prime}:H\to L^{2}\left(\mu^{\prime}\right)$,
such that $Y$ factors through the tautological embedding $\id:L^{2}\left(\mu^{\prime}\right)\to L^{0}\left(\mu\right)$
as follows:
\[
\xymatrix{H\ar[rr]^{Y}\ar@{-->}[dr]_{Y^{\prime}} &  & L^{0}\left(\mu\right)\\
 & L^{2}\left(\mu^{\prime}\right)\ar[ur]_{\id}
}
\]
\label{thm:Factorization}
\end{thm}
The factorization theorem is the ultimate reason behind all occurrences
of ``an equivalent measure $\mu^{\prime}\sim\mu$'' throughout the
text.

Using this theorem we show that the definitions of Gaussian fields
in terms of a couple $\left(X,Y\right)$ and a jointly Gaussian family
of integrals against $L^{2}\left(\mu^{\prime}\right)$ test functions
are equivalent.

Let $X$ and $Y$ be a standard Gaussian in $H$ and a generalized
$H$-valued function defined on $\left(\mathcal{T},\mu\right)$, respectively.
Then by Theorem \ref{thm:Factorization} there exists an equivalent
measure $\mu^{\prime}\sim\mu$, such that
\begin{equation}
\forall\xi\in H:\intop\left|\left\langle Y\left(t\right),\xi\right\rangle \right|^{2}\mu^{\prime}\left(dt\right)<\infty.\label{eq:FactorizationY}
\end{equation}
Now fix some function $f\in L^{2}\left(\mu^{\prime}\right)$. It follows
from \eqref{eq:FactorizationY} that for every $\xi\in H$ the function
$f\cdot\left\langle Y,\xi\right\rangle $ is in $L^{1}$. In other
words, $fY$ has a weak first moment $\E_{\mu^{\prime}}fY$. Thus
one can define for every test function $f\in L^{2}\left(\mu^{\prime}\right)$
a Gaussian random variable $\left\langle X,\E_{\mu^{\prime}}fY\right\rangle $,
which is to be interpreted as the integral of the ``Gaussian field
$\left(\left\langle X,Y\left(t\right)\right\rangle \right)$'' against
the test function $f$.

Conversely, suppose that we have a map that takes any test function
$f\in L^{2}\left(\mu^{\prime}\right)$ to a measurable linear functional
of some Gaussian vector $X$, i.e. a variable of the form $\left\langle X,Af\right\rangle $
for some $Af\in H$, such that the operator $A:L^{2}\left(\mu^{\prime}\right)\to H$
is bounded. Then one can define a generalized random vector $Y$ as
the composition $H\overset{A^{\ast}}{\to}L^{2}\left(\mu^{\prime}\right)\overset{\id}{\to}L^{0}\left(\mu\right)$.

Thus we have constructions that produce a couple $\left(X,Y\right)$
from a field defined by test functions and vice versa. The verification
that they are inverse to each other reduces to the routine unraveling
of the definitions which is left to the reader.

\bibliographystyle{abbrv}
\bibliography{MultiplicativeChaos}

\begin{thebibliography}{10}

\bibitem{AKQRandomPolymer}
T.~Alberts, K.~Khanin, and J.~Quastel.
\newblock {The continuum directed random polymer}.
\newblock {\em Journal of Statistical Physics}, 154(1-2):305–326, 2014.

\bibitem{BJRVKPZ}
J.~Barral, X.~Jin, R.~Rhodes, and V.~Vargas.
\newblock {Gaussian multiplicative chaos and KPZ duality}.
\newblock {\em Communications in Mathematical Physics}, 323(2):451–485, 2013.

\bibitem{BKNSWCritical}
J.~Barral, A.~Kupiainen, M.~Nikula, E.~Saksman, and C.~Webb.
\newblock {Critical mandelbrot cascades}.
\newblock {\em Communications in Mathematical Physics}, 325(2):685–711, 2014.

\bibitem{Bog}
V.~I. {Bogachev}.
\newblock {\em {Gaussian measures}}.
\newblock Providence, RI: American Mathematical Society, 1998.

\bibitem{DSRVRenormalization}
B.~Duplantier, R.~Rhodes, S.~Sheffield, and V.~Vargas.
\newblock {Renormalization of Critical Gaussian Multiplicative Chaos and KPZ
  Relation}.
\newblock {\em Communications in Mathematical Physics}, 330(1):283–330.

\bibitem{DSRVCriticalConvergence}
B.~Duplantier, R.~Rhodes, S.~Sheffield, and V.~Vargas.
\newblock {Critical Gaussian multiplicative chaos: Convergence of the
  derivative martingale}.
\newblock {\em The Annals of Probability}, 42(5):1769–1808, 2014.

\bibitem{DuShKPZ}
B.~Duplantier and S.~Sheffield.
\newblock {Liouville quantum gravity and KPZ}.
\newblock {\em Inventiones mathematicae}, 185(2):333–393, 2010.

\bibitem{HuShiMinimal}
Y.~Hu and Z.~Shi.
\newblock {Minimal position and critical martingale convergence in branching
  random walks, and directed polymers on disordered trees}.
\newblock {\em The Annals of Probability}, 37(2):742–789, 2009.

\bibitem{Janson}
S.~Janson.
\newblock {\em {Gaussian Hilbert Spaces}}.
\newblock Cambridge University Press, 1997.
\newblock Cambridge Books Online.

\bibitem{KahaneSur}
J.-P. Kahane.
\newblock {Sur le chaos multiplicatif}.
\newblock {\em Annales des sciences mathématiques du Québec}, 9(2):105–150,
  1985.

\bibitem{LRVComplex}
H.~Lacoin, R.~Rhodes, and V.~Vargas.
\newblock {Complex Gaussian multiplicative chaos}.
\newblock {\em Communications in Mathematical Physics}, 337(2):569–632, 2015.

\bibitem{MRVSupercritical}
T.~Madaule, R.~Rhodes, and V.~Vargas.
\newblock {Glassy phase and freezing of log-correlated Gaussian potentials}.
\newblock {\em arXiv preprint arXiv:1310.5574}, 2013.

\bibitem{Maurey1}
B.~Maurey.
\newblock {Théorèmes de Nikishin: théorèmes de factorisation pour les
  applications linéaires à valeurs dans un espace $L^0 (\Omega , \mu )$}.
\newblock {\em Séminaire Analyse fonctionnelle (dit "Maurey-Schwartz")}, page
  1–10, 1972-1973.

\bibitem{Maurey2}
B.~Maurey.
\newblock {Théorèmes de Nikishin: théorèmes de factorisation pour les
  applications linéaires à valeurs dans un espace $L^0 (\Omega , \mu )$
  (suite et fin)}.
\newblock {\em Séminaire Analyse fonctionnelle (dit "Maurey-Schwartz")}, page
  1–8, 1972-1973.

\bibitem{Nikishin}
E.~Nikishin.
\newblock {Resonance Theorems and Superlinear Operators}.
\newblock {\em Russian Mathematical Surveys}, 25(6):125, 1970.

\bibitem{RVReview}
R.~Rhodes, V.~Vargas, et~al.
\newblock {Gaussian multiplicative chaos and applications: a review}.
\newblock {\em Probability Surveys}, 11, 2014.

\bibitem{RVRevisited}
R.~Robert and V.~Vargas.
\newblock {Gaussian multiplicative chaos revisited}.
\newblock {\em The Annals of Probability}, 38(2):605–631, 2010.

\bibitem{STTranslation}
H.~Sato and M.~Tamashiro.
\newblock {Multiplicative chaos and random translation}.
\newblock {\em Annales de l'institut Henri Poincaré, section B},
  30(2):245–264, 1994.

\end{thebibliography}

\end{document}